\documentclass[11pt]{article}

\usepackage{latexsym,mathrsfs}
\usepackage{amsmath,amssymb} 
\usepackage{amsthm,enumerate,verbatim}
\usepackage{amsfonts}
\usepackage{graphicx}
\usepackage[Algorithm,ruled]{algorithm}
\usepackage{url}
\usepackage{color}
\usepackage{cite}
\usepackage{float}
\usepackage{wrapfig}
\usepackage[normalem]{ulem}
\usepackage{capt-of}

\usepackage{algorithmicx}
\usepackage{algpseudocode}

\usepackage{enumitem}
\usepackage{algcompatible}
\usepackage{enumerate}

\newcommand{\mc}{\mathscr}
\newcommand{\mca}{\mathcal}
\newcommand{\cu}{\subseteq}
\newcommand{\f}{\mathbb}
\newcommand{\ol}{\overline}
\newcommand{\wt}{\widetilde}
\newcommand{\bm}{\mathbf}
\newcommand{\R}{\mathbb R}
\newcommand{\SV}{{\rm SV}}
\newcommand{\MA}{{\rm MA}}
\newcommand{\PSV}{{\rm PSV}}

\newtheorem{lemma}{Lemma}
\newtheorem{proposition}{Proposition}

\newtheorem{theorem}{Theorem}
\newtheorem{corollary}{Corollary}

\newtheorem*{definition*}{Definition}

\newtheorem{remark}{Remark}

\DeclareMathOperator{\argmin}{argmin} 
 
\DeclareMathOperator{\rank}{rank}

\DeclareMathOperator{\tr}{tr}

\definecolor{brightpink}{rgb}{1.0, 0.0, 0.5}

\title{Computing the least cone-constrained singular value  
of matrices} 
\date{}

\author{Giovanni Barbarino\thanks{University of Mons, Rue de Houdain 9, 7000 Mons, Belgium.  GB is member of the Research Group GNCS (Gruppo Nazionale per il
Calcolo Scientifico) of INdAM (Istituto Nazionale di Alta Matematica). 
Email: giovanni.barbarino@gmail.com. } \and Nicolas Gillis\thanks{University of Mons, Rue de Houdain 9, 7000 Mons, Belgium.  
Email: nicolas.gillis@umons.ac.be. } 
	\and David Sossa\thanks{Universidad de O'Higgins, Instituto de Ciencias de la Ingenier\'ia, Av.\,Libertador Bernardo O'Higgins 611, Rancagua, Chile. E-mail: david.sossa@uoh.cl.}} 

\begin{document}

\maketitle

\begin{abstract}

This paper deals with the numerical computation of the least singular value of a rectangular matrix $A$ relative to a pair of closed convex cones $(P,Q)$, which is defined as the optimal value of the non-convex optimization problem of minimizing $\langle u,Av\rangle$ such that $u$ and $v$ are unit vectors in $P$ and $Q$, respectively. When $A$ is the identity matrix, the least singular value coincides with the cosine of the largest angle between $P$ and $Q$. When $P$ and $Q$ are positive orthants, the least singular value is called the least Pareto singular value of $A$ and has applications, for instance, in graph theory. We prove the NP-hardness of all the above problems, while identifying cases when such problems can be solved in polynomial time. 
We then propose four algorithms. Two are exact algorithms,  meaning that they are guaranteed to compute a globally optimal solution; one uses an exact non-convex quadratic programming solver, and the other a brute-force active-set method. 
The other two are heuristics, meaning that they rapidly compute locally optimal solutions; one uses an alternating projection algorithm with extrapolation, and the other a sequential partial linearization approach based on fractional programming. 
We illustrate the use of these algorithms on several examples. 
\end{abstract}

\noindent\textbf{Keywords.} Cone-constrained singular values, maximal angles between cones, Pareto singular values, non-convex problem, complexity.

\noindent\textbf{AMS subject classification.} 15A18, 90C26, 
68Q15, 65K05 

\bigskip
\noindent Communicated by Luis Zuluaga.

\section{Introduction}
 Singular values of matrices are ubiquitous in applied linear algebra, and they are at the core of essential tools in data analysis such as least-square techniques, principal component analysis, and principal angles of subspaces. 
 For $A\in \R^{m\times n}$, its singular values are obtained by computing the critical (stationary) pairs of the problem of minimizing $\langle u,Av\rangle$ subject to $u\in S_m$ and $v\in S_n$, where $S_d$ denotes the unit sphere in $\R^d$. Recently, in the series of papers \cite{seeger2023SV,seeger2023PSV,seeger2023SVLM}, Seeger and Sossa have considered the study of such optimization problems  but with the extra condition that $u$ and $v$ range on the closed convex cones $P\subseteq\mathbb R^m$ and $Q\subseteq\mathbb R^n$, respectively. That is, they investigated the critical values of the \emph{least cone-constrained singular value problem}: 
\begin{equation}\label{eq:SVcones}
\SV(A,P,Q):\;\;\min_{u,v}\langle u,Av\rangle\quad\text{s.t.}\quad u\in P\cap S_m,\,v\in Q\cap S_n.
\end{equation} 
The critical values of $\SV(A,P,Q)$ are obtained from the KKT optimality conditions to \eqref{eq:SVcones}. More precisely,  a real number $\sigma$ is a \emph{critical value} of $\SV(A,P,Q)$ if there exists $(u,v)\in S_m\times S_n$ such that
\[P\ni u\perp (Av-\sigma u)\in P^\ast\quad\text{and}\quad Q\ni v\perp (A^\top u-\sigma v)\in Q^\ast,\]
where $\perp$ denotes orthogonality, and $P^\ast$ denotes the dual cone of $P$.
The critical values of $\SV(A,P,Q)$ give us the \emph{singular values of $A$ relative to $(P,Q)$} which we will refer to as the  $(P,Q)$-singular values of $A$, and the corresponding pairs
$(u,v)$ will be referred to as
\emph{critical pairs} of $\SV(A,P,Q)$. 
We say that $(u,v)$ is a \emph{solution pair} of \eqref{eq:SVcones} if it solves \eqref{eq:SVcones}. Observe that when $P=\mathbb R^m$ and $Q=\mathbb R^n$, the critical values of \eqref{eq:SVcones} correspond to
the (classical) singular values of $A$.

Seeger and Sossa \cite{seeger2023SV,seeger2023PSV,seeger2023SVLM} focused on the study of $\SV(A,P,Q)$ mainly from a theoretical point of view. For instance, they studied stability and cardinality issues of the $(P,Q)$-singular value spectrum of $A$. They proved that the cardinality of this spectrum is finite whenever $P$ and $Q$ are polyhedral cones. 
They also showed that $\SV(A,P,Q)$ covers many interesting optimization problems, including maximal angle between two cones \cite{iusem-seeger2005maximal,seeger-sossa2016angle1,seeger-sossa2016angle2,orlitzky2020maximal,de2023computing}, cone-constrained principal component analysis \cite{montanari-richard2016pca,deshpande-montanari2014conepca}, and nonnegative rank-one matrix factorization  \cite{gillis2014biclique}.

When $A$ is the identity matrix, $A = I \in \R^{n\times n}$,  
 $\SV(I,P,Q)$ becomes the problem of computing the \emph{maximal angle between two cones}: 
\begin{equation}\label{eq:angles_cones}
\MA(P,Q):\;\;\min_{u,v}\langle u,v\rangle\quad\text{s.t.}\quad u\in P\cap S_n,\,v\in Q\cap S_n.
\end{equation}
The optimal value of $\MA(P,Q)$ is the cosine of the maximal angle between $P$ and $Q$. The arccosine of the critical values of $\MA(P,Q)$ are called the \emph{critical angles} between $P$ and $Q$. The theory of critical angles is discussed in \cite{iusem-seeger2005maximal, seeger-sossa2016angle1, seeger-sossa2016angle2}, some numerical methods are presented in \cite{orlitzky2020maximal,de2023computing}, and an application to an image set classification problem is given in \cite{sogi2022imagesetangle}.

Another interesting case is when $P$ and $Q$ are the nonnegative orthants; that is, $P=\mathbb R^m_+$ and $Q=\mathbb R^n_+$. In this case, the $(\mathbb R^m_+,\mathbb R^n_+)$-singular values of $A$ are called the \emph{Pareto singular values} of $A$, and $\SV(A,\R^m_+,\R^n_+)$ becomes the \emph{least Pareto singular value problem}:
\begin{equation}\label{eq:pareto}
\PSV(A):\;\;\min_{u,v}\langle u,Av\rangle\quad\text{s.t.}\quad u\in \R^m_+\cap S_m,\,v\in \R^n_+\cap S_n.
\end{equation}
Pareto singular values have applications in spectral graph theory. For instance, in \cite{seeger-sossa2024-boolean}, Pareto singular values of Boolean matrices were studied for analyzing structural properties of bipartite graphs.

To the best of our knowledge, there are only a few works about the numerical resolution of $\SV(A,P,Q)$. 
For example, for $\PSV(A)$, it is possible to find all the Pareto singular values of $A$ by 
brute force computation. Unfortunately, this is only possible when $m$ and $n$ are small numbers, 
see~\cite{seeger2023PSV}. Recently, an efficient numerical method was proposed for the maximal angle problem $\MA(P,Q)$~\cite{de2023computing} 
by reformulating $\MA(P,Q)$ as a fractional program.

\paragraph{Contribution and outline of the paper} 

In this work, we study the numerical computation of the cone-constrained singular value problem. 
We first prove, in Section~\ref{sec:nphard}, 
 that $\SV(A,P,Q)$, $\MA(P,Q)$ and $\PSV(A)$ are NP-hard problems. In Section~\ref{sec:simplecases}, we identify several cases when these problems can be solved in polynomial time, under appropriate conditions on the input, $A$, $P$ and $Q$. 
 In Section~\ref{sec:algorithms}, we first explain how 
to check whether a given problem can be solved in polynomial time using the results from Section~\ref{sec:simplecases}, then we propose four algorithms. Two are exact algorithms presented in Section~\ref{sec:exactalgo},  meaning that they are guaranteed to compute a globally optimal solution; one uses a brute-force active-set method (Section~\ref{sec:bruteforce}), and the other an exact non-convex quadratic programming solver (Section~\ref{sec:gurobi}). 
The other two are heuristics presented in Section~\ref{sec:heuristics}, meaning that they rapidly compute locally optimal solutions; one uses an alternating projection algorithm with extrapolation (Section~\ref{sec:AO}), and the other a sequential partial linearization approach based on fractional programming (Section~\ref{sec:fracprog}). 
We illustrate the use of these algorithms on several examples in Section~\ref{sec:numexp}.

\paragraph{Notation} 

In $\R^d$, we denote $\langle\cdot,\cdot\rangle$ the standard inner product and $\Vert\cdot\Vert$ its induced norm (Euclidean norm). Let $S_d:=\{x\in\R^d:\Vert x\Vert=1\}$, and denote $\R^{m\times n}$ the space of $m\times n$ real matrices. For a matrix $M$, $M \geq 0$ means that $M$ is  component-wise nonnegative.  

Given a cone $P\cu \f R^m$ and a matrix $M\in \f R^{\ell\times m}$, then $MP$ is the cone in $\f R^\ell$ defined as $MP:= \{ Mx\ |\ x\in P \}$. 
Unless specified otherwise, in this paper, we work with polyhedral cones, $P$ and $Q$, which are convex, closed and finitely generated. Hence, without loss of generality (w.l.o.g.), we assume that
\begin{equation}\label{polyhedral}
P=G \mathbb R^p_+ = \{ Gx \ | \ x \geq 0 \} 
\quad \mbox{and} \quad 
Q=H \mathbb R^q_+  = \{ Hy \ | \ y \geq 0 \}, 
\end{equation}
where $G=[g_1,\ldots,g_p]$ and $H=[h_1,\ldots,h_q]$ are matrices whose columns are {conically independent vectors}, that is, a minimal set of conic generators. 
W.l.o.g., we assume the columns of $G$ and $H$ are unit vectors, that is, $\|g_i\| = 1$ for all $i$ and $\|h_j\| = 1$ for all $j$.  
The \textit{faces} of $P$ are the polyhedral cones $F\cu P$ such that  
\[
v_1 +  v_2\in F, \quad v_1,v_2\in P  \implies v_1,v_2\in F.
\]
In particular, any face is generated by some subset of the columns of $G$. 
The dimension of a face is equal to the dimension of the subspace generated by its generators.

The dual of a cone $P$ is defined as $P^* = \{ x \ | \langle x,y\rangle \geq 0,\,\forall y \in P\}$. 
A \textit{ray} of $P$ is any nonzero vector of a $1$-dimensional face of $P$,  or equivalently any positive multiple of a column of $G$.
A face is called \textit{proper} if it is not equal to $\{0\}$ or to the whole cone $P$. 
The \textit{facets} of the cone $P$ are the proper faces of maximal dimension. 
Given a vector $v$, we denote as $\f R_+ v$ the cone of the nonnegative multiples of $v$, that is,  $\f R_+v:= \{\alpha v \ | \ \alpha \ge 0 \}$.

For a positive integer number $n$, we denote $[n] = \{1,2,\dots,n\}$. Moreover, given a subset $\mathcal I\cu [n]$ and a vector $x\in \f R^n$, we denote as $x_{\mathcal I}   \in \f R^{|\mathcal I|}$ the restricted vector to the indices contained in $\mca I$, that is, if $\mca I = \{ i_1,i_2,\dots, i_{|\mca I|}\}$ and the $i$-th entry of $x$ is $x_i$, then $x_{\mca I} = [x_{i_1}, x_{i_2},\dots,x_{ i_{|\mca I|}}]^\top$. Analogously, if $M \in \f R^{m\times n}$ is a matrix with $n$ columns, then $M_{:,\mca I}\in \f R^{m\times |\mca I|}$ is a shortcut for the matrix restricted to the columns with indices in $\mca I$. We denote by $\mca I^c$ the set of indices of $[n]$ that are not in $\mca I$.

We denote as $\rho(C)$ the spectral radius of a square matrix $C$. In $\R^{m\times n}$, 
the norm $\Vert\cdot\Vert_F$ denotes the Frobenius norm, that is, $\Vert A\Vert_F=\sqrt{{\rm trace}(A^\top A)}$, and $\| \cdot \|$ denotes the spectral norm, that is, $\| A \|=\max_{\Vert z\Vert=1}\Vert Az\Vert$ (largest singular value of $A$). 
The matrix $I_n$ is the identity of dimension $n$, and we will simply use  $I$ when the dimension is clear from the context.
Finally, ${\bf 1}$ is the all-ones vector and $e_i$ is the $i$-th unit vector, both of the appropriate dimension that will be clear from the context.

\section{Computational complexity} \label{sec:nphard}

In this section, we discuss the computational complexity of the various problems introduced in the previous section, namely $\SV(A,P,Q)$, $\MA(P,Q)$ and $\PSV(A)$ when $P,Q$ are generated by matrices $G,H$. 
To the best of our knowledge, this question has not been addressed previously in the literature. 
We prove NP-hardness of $\PSV(A)$ (Theorem~\ref{th:SVcones_NPhard}), which implies NP-hardness of $\SV(A,P,Q)$ since $\PSV(A) = \SV(A,\R^m_+,\R^n_+)$. 
Interestingly, we prove that any problem $\SV(A,P,Q)$ can be reduced in polynomial time to a problem $\MA(\wt P,\wt Q)$ (Theorem~\ref{th:reduction_SVcones_anglescones} and Lemma~\ref{lem:orthogonal_rational_decomposition}), which in turn implies the NP-hardness of $\MA(P,Q)$ (Corollary~\ref{cor:angles_cones_NPhard}).

\begin{remark}[Hardness results]
    In this paper, we focus on the NP-hardness of the studied optimization problems, namely SV$(A,P,Q)$, MA$(P,Q)$ and PSV$(A)$. 
    For simplicity of the presentation, we do not consider approximate versions nor decision versions of these problems, nor do we check whether the optimal solutions of these problems are polynomial in the size of the input. This is a topic for further research. 
\end{remark}

\subsection{Least Pareto singular value is NP-hard}

In general, the problem $\SV(A,P,Q)$ is difficult to solve since it is nonconvex. Indeed, we prove in this section that  $\PSV(A) = \SV(A,\R^m_+,\R^n_+)$ is NP-hard. To do so, we rely on the following result. 

\begin{theorem}\cite[Corollary 1]{gillis2014biclique} \label{th:rank1NF}
    Given $M \in \R^{m\times n}$, it is NP-hard to solve  
\begin{equation} \label{eq:rank1NF}
\min_{u \in\mathbb R^m, v\in \R^n} \| M - uv^\top \|_F^2 
\quad  
\text{ such that } \quad 
u\geq0,\;v\geq 0. 
\end{equation}
\end{theorem}
The proof of Theorem~\ref{th:rank1NF} relies on the maximum edge biclique problem (MEBP) in a bipartite graph: Given a binary biadjacency matrix $B \in \{0,1\}^{m \times n}$ where $B(i,j) = 1$ if node $i$ and $j$ are connected (one on each side of the bipartite graph), the goal is to find binary vectors $u$ and $v$ such that $u^\top B v$ is maximized while $u_i v_j \leq B(i,j)$ for all $i,j$. The vectors $(u,v)$ indicate the position of the largest bipartite subgraph where all nodes are connected (which is called a biclique).
The proof of Theorem~\ref{th:rank1NF} shows that the MEBP can be reduced to  \eqref{eq:rank1NF} with $M(i,j) = B(i,j) - (1-B(i,j))d$ for all $(i,j)$ and any $d \geq \max(m,n)$. The matrix $M$ replaces the zero entries of $B$ by $-d$'s in order to enforce these entries to be approximated by zeros in the optimal solutions of~\eqref{eq:rank1NF}.  

Note that, when $M \geq 0$, the problem can be solved in polynomial time by the Perron-Frobenius and Eckart-Young theorems. Hence, allowing negative entries in $M$ is crucial for the NP-hardness.

\begin{theorem}\label{th:SVcones_NPhard}
    Computing the least Pareto singular values is NP-hard, that is,  $\PSV(A)$ is NP-hard. 
\end{theorem}
\begin{proof}
Let $M$ be a $m\times n$ matrix with at least one positive entry. The result \cite[Theorem\,3]{seeger2023PSV} shows that solving \eqref{eq:rank1NF} is equivalent to solving $\PSV(-M)$, thus both problems are NP-hard. 
\end{proof}

As $\PSV(A)$ is a particular instance of $\SV(A,P,Q)$, we have the following corollary. 
\begin{corollary}\label{cor:ConeSV_is_NPHard}
Computing the least cone-constrained singular value, that is, $\SV(A,P,Q)$, is NP-hard.
\end{corollary}

\subsection{Maximum conic angle is NP-hard}

 The problem $\MA(P,Q)$ of finding the maximal angle between the cones, $P$ and $Q$, formulated in \eqref{eq:angles_cones},  is also a nonconvex problem. Although being a particular case of $\SV(A,P,Q)$, with $A = I_n$, it can be proved that any algorithm computing the maximal angle between cones can also solve the more general problem $\SV(A,P,Q)$ for any matrix $A\in \f R^{m\times n}$.  
\begin{lemma}
\label{lem:orthogonal_decomposition}  
Any matrix $A\in \f R^{m\times n}$ of spectral norm 1 and $m\ge n$ can be decomposed as $A=U^\top V$ where 
{$U\in \f R^{(m+s)\times m}$ and $V\in \f R^{(m+s)\times n}$} are matrices with orthonormal columns, and $s$ is the rank of $A^\top A-I$.
\end{lemma}
\begin{proof}
  Since $I_n-A^\top A$ has rank $s$ and is positive semidefinite, it admits a Cholesky decomposition $I_n-A^\top A = LL^\top $ with $L\in \f R^{n\times s}$. 
Define $U\in \f R^{(m+s)\times m}$ and $V\in \f R^{(m+s)\times n}$ as  $U = \begin{pmatrix}
    I_m\\ 0
\end{pmatrix}$, $V := \begin{pmatrix}
    A\\L^\top 
\end{pmatrix},$ 
so that $ V^\top V =  A^\top A + LL^\top  = 
  I_n$, that is, all columns of $V$ are orthogonal to each other, have unit norm, and $A =  U^\top V.$
\end{proof}

With the last result, we can now prove the equivalence between the computation of the least $(P,Q)$-singular values and the maximal angle between two cones. 

\begin{theorem}
   \label{th:reduction_SVcones_anglescones}   Let $A\in\mathbb R^{m\times n}$ be nonzero, $m\geq n$, and $\Vert A\Vert^{-1}A=U^\top V$ be the decomposition given in Lemma \ref{lem:orthogonal_decomposition}. Then,
 $(u^*,v^*)$ solves $\SV(A,P,Q)$ if and only if $(Uu^*,Vv^*)$ solves $\MA(\wt P,\wt Q)$ where $\wt P:=UP$ and $\wt Q:=VQ$. 
\end{theorem} 
\begin{proof}

Since $U$ and $V$ have orthonormal columns, $\SV(A,P,Q)$ can be transformed as follows: 
\begin{eqnarray*}
   \lambda^*&:=&\min_{\scriptsize \begin{array}{l}u\in P\cap S_m,\\v\in Q\cap S_n\end{array}}\langle u,Av\rangle = \Vert A\Vert
    \min_{\scriptsize \begin{array}{l}u\in P\cap S_m,\\v\in Q\cap S_n\end{array}}\langle Uu,Vv\rangle =
     \Vert A\Vert\min_{\scriptsize \begin{array}{l}x\in \wt P\cap S_{m+s},\\y\in \wt Q\cap S_{m+s}\end{array}}\langle x,y\rangle 
     \;=:\;\delta^*,
\end{eqnarray*}
where $\wt P=UP$ and $\wt Q=VQ$ are closed convex cones in $\f R^{m+s}$ that are isometrically equivalent to $P,Q$.  Now, suppose that $(u^*,v^*)$ solves $\SV(A,P,Q)$. Then $\lambda^*=\langle u^*,Av^*\rangle=\Vert A\Vert\langle Uu^*,Vv^*\rangle \geq \delta^*=\lambda^*,$
where the second equality holds because $\Vert A\Vert^{-1}A=U^\top V$, and the inequality holds because $(Uu^*,Vv^*)$ is feasible for $\MA(\wt P,\wt Q)$. Hence, $(Uu^*,Vv^*)$ solves $\MA(\wt P,\wt Q)$. The converse is analogous.
\end{proof}

\begin{remark}[Further connection between $\SV(A,P,Q)$ and $\MA(\wt P,\wt Q)$] 
The proof of the above theorem also draws the following connection between problems $\SV(A,P,Q)$ and $\MA(\wt P,\wt Q)$: $(u,v)$ is a critical pair of $\SV(A,P,Q)$ if and only if $(Uu,Vv)$ is a critical pair between $\wt P$ and $\wt Q$. Furthermore, $\sigma$ is a $(P,Q)$-singular value of $A$ if and only if $\arccos(\sigma)$ is a critical angle between $\wt P$ and $\wt Q$.
\end{remark}

In order to prove that $\SV(A,P,Q)$ reduces in polynomial time to $\MA(\wt P,\wt Q)$ for opportune $\wt P$ and $\wt Q$ with dimensions and generators that can be expressed polynomially in the dimensions and generators of $A$, $P$ and $Q$, we need a different decomposition $A=cU^\top V$ that preserves the rationality of the matrices when $A\in \f Q^{m\times n}$. Notice that the argument of Theorem \ref{th:reduction_SVcones_anglescones} does not produce a polynomial reduction since  $\|A\|$ might be irrational and the matrices $U,V$ might have irrational entries. 
\begin{lemma}
\label{lem:orthogonal_rational_decomposition}  
Any matrix $A\in \f Q^{m\times n}$ with $m\ge n$ can be decomposed as $A=cU^\top V$ where  
{$U\in \f Q^{(m+\binom{n}{2}+t)\times m}$ and $V\in \f Q^{(m+\binom{n}{2}+t)\times n}$} are matrices with orthonormal columns, and $c\in\f Q$. 
The parameter $t$, all the entries of $U,V$, and the computational time of the decomposition are all polynomial in $m,n$ and in the largest number of bits used to represent an entry in $A$.  
\end{lemma}
\begin{proof}
Let $\ol a$ be the largest number of bits used to represent an entry in $A$, and call $a_1,\dots,a_n$ the columns of $A$.
Let $B\in \f Q^{\binom n2\times n}$ be a matrix with $\binom n2$ rows, one for each pair $(a_i,a_j)$ of columns of $A$ with $i< j$, and suppose that such row is zero apart from the $i$-th entry equal to $1$ and the $j$-th entry equal to $-\langle a_i,a_j\rangle$. Notice that each entry $-\langle a_i,a_j\rangle$  can be expressed with $\mathcal O(\ol a + m)$ bits and it can be computed in $\mathcal O(m\ol a^2)$ time.  

If now we call $Q:= \begin{pmatrix}
    A\\B
\end{pmatrix}$, we find that its columns $q_i$ are pairwise orthogonal. Let  $d$ be the product of the denominators of the fractions $\|q_i\|^2\in\f Q$. 
The matrix $dQ$ has orthogonal columns, whose norm squared are now $p_i:=d^2\|q_i\|^2\in \f N$. 
Let $\ell\in\f N$ be the least integer such that $p_i<\ell^2$ for every $i$, and call $n_i:= \ell^2 - p_i\in\f N$. 
Each  $\|q_i\|^2$ can be computed in $\mathcal O(m\ol a^2 + n(\ol a + m)^2 )= \mathcal O(m(\ol a + m)^2)$ time and can be expressed with $\mathcal O(\ol a + m + n ) = \mathcal{O}(\ol a + m)$ bits, and the same can be said for $d$, $p_i$, $\ell$ and $n_i$. 

Write each $n_i$ as a sum of $t_i$ perfect squares by following the standard greedy approach: find the greatest integer $n_{i,1}$ for which $n_i = {n_{i,1}}^2 + m_i$, with $m_i\ge 0$ and repeat on $m_i$ until it reaches zero. Since for $n_i\ge 9$ one has $m_i\le n_i/2$,  there are $t_i =\mathcal O(\log(n_i))= \mathcal{O}(\ol a + m)$ perfect squares $n_{i,1}, \dots, n_{i,t_i}$ in the decomposition and it is possible to compute them all in $\mathcal O(\log(n_i)^3)$ time. Notice that $n_{i,j}\le n_i$, so they are also representable with $\mathcal{O}(\ol a + m)$ bits.

Let now $N\in \f Q^{t\times n}$ be a matrix with $t = \sum_i t_i$ rows, where for each index $i$ we have the rows $n_{i,1}e_i^\top , \dots, n_{i,t_i}e_i^\top$. Define $U\in \f Q^{(m+\binom n2+t)\times m}$ and $V\in \f Q^{(m+\binom n2+t)\times n}$ as  $U = \begin{pmatrix}
    I_m\\ 0
\end{pmatrix}$, $ V := \frac 1{\ell}\begin{pmatrix}
    dA\\dB\\N
\end{pmatrix}$. By construction, we have that $U,V$ have orthonormal columns and $\frac \ell d U^\top V =  A$. 

Eventually, notice that the parameter $t$ can be expressed with $\mathcal O(\log(n(\ol a + m )))$ bits,
whereas the parameter $c = \ell/d$ and each entry of $V$  can all be expressed with $ \mathcal{O}(\ol a + m)$ bits.
The computational cost of the whole decomposition is thus $\mathcal O(\ol a^2 mn^2 + m(\ol a + m)^3)$.
\end{proof}
Combining the proof of Theorem \ref{th:reduction_SVcones_anglescones} with the new decomposition given by Lemma \ref{lem:orthogonal_rational_decomposition}, we find that  $\SV(A,P,Q)$ reduces in polynomial time to $\MA(\wt P,\wt Q)$ where $\wt P:=UP$ and $\wt Q:=VQ$.

Since  $\SV(A,P,Q)$ is NP-hard (Corollary \ref{cor:ConeSV_is_NPHard}) and it reduces to $\MA(P,Q)$, 
 we get the following corollary.

\begin{corollary}
      \label{cor:angles_cones_NPhard}  Computing the maximal angle between two polyhedral cones $\MA(P,Q)$ is NP-hard.  
\end{corollary}

\section{Some cases solvable in polynomial time} \label{sec:simplecases}

Although the problems considered in this paper are NP-hard in general, as proved in the previous section, 
it turns out that there are a few interesting cases when these problems can be solved in polynomial time. This is the focus of this section. This will be useful when designing algorithms in the next section, allowing us to check beforehand whether a given problem can be solved easily.

Let us define the \textit{Nash pairs} $(u^*, v^*)\in P\times Q$ of the problem $\SV(A,P,Q)$ as the  vectors such that
\[
\langle u, A v^*\rangle \ge \langle u^*, A v^*\rangle, \qquad \langle u^*, A v\rangle \ge \langle u^*, A v^*\rangle, \qquad \forall\, (u, v)\in P\times Q : \|u\| = \|v\| = 1.
\]
Observe that the above means that $u^*$ solves $\SV(A,P,\mathbb R_+ v^*)$ and $v^*$ solves $\SV(A,\R_+ u^*,Q)$. Observe also that any solution pair and all local minima with nonpositive values of $\SV(A,P,Q)$ are Nash pairs.

\subsection{Solving $\SV(A,P,Q)$ when $G^\top AH\ge 0$} 

When $G^\top AH\ge 0$,  it is easy to identify all Nash pairs and, by extension, all solution pairs. 
\begin{proposition}\label{pro:nonnegative_case_SVcones}
    Let $P$ and $Q$ be polyhedral cones generated by $G$ and $H$, respectively, as in~\eqref{polyhedral}. 
    Suppose that $G^\top A H$ is a nonnegative matrix.
    Then, there exists a solution pair $(u,v)$ of $\SV(A,P,Q)$ where $u$ and $v$ are generators 
    of $P$ and $Q$, respectively, and the optimal value of $\SV(A,P,Q)$ is the minimum entry of $G^\top A H$. Moreover, if $G^\top A H$ is strictly positive, then any Nash pair is achieved by a couple of generators in $P$ and $Q$.
\end{proposition}
\begin{proof}
    Notice that since the columns of $G$ and $H$ have unit norm, we have,
    \[
    0 \quad \le \quad  \lambda^*:=\min_{\scriptsize \begin{array}{l}u\in P\cap S_m,\\v\in Q\cap S_n\end{array}}\langle u,Av\rangle = 
      \min_{\scriptsize \begin{array}{l}x\ge0 ,\,\Vert Gx\Vert=1,\\y\ge 0,\,\Vert Hy\Vert=1\end{array}}\langle G x,  AHy\rangle  \quad  \le   \quad  \min_{i,j}\langle g_i, A h_j\rangle. 
    \]
    Thus, if $G^\top A H$ has a zero entry then $(G^\top A H)_{i,j}=\langle g_i, A h_j\rangle=0$ for some $i,j$. From the above relation, we get $\lambda^*=0$, and the thesis is proved with $u = g_i$ and $v = h_j$. 
    
    Suppose now that $G^\top A H>0$ which means that $\min_{i,j}\langle g_i, A h_j\rangle>0$. Notice that this implies $\langle u, A v\rangle = \langle Gx,  AHy\rangle>0$ for all nonzero and nonnegative $x,y$, that is, any nonzero $u,v$ in the respective cones. Suppose that $(u,v)$ is a Nash pair and that $u = \alpha_1 u_1 + \alpha_2 u_2$ with $u_1,u_2\in P$, $\|u\| = \|u_1\| = \|u_2\| = 1$, and $\alpha_1,\alpha_2>0$. Then 
    \[
    \langle u, Av \rangle = \alpha_1\langle u_1, Av\rangle + \alpha_2\langle u_2, Av\rangle\ge (\alpha_1 + \alpha_2) \langle u, Av\rangle\implies 1 \ge \alpha_1+\alpha_2, 
    \]
and, at the same time, 
\[
1 = \|u\|^2 = \|\alpha_1 u_1 + \alpha_2 u_2\|^2 = \alpha_1^2 + \alpha_2^2 +2\alpha_1\alpha_2\langle u_1, u_2\rangle\le 1 + 2\alpha_1\alpha_2(\langle u_1, u_2\rangle-1)
\]
\[\implies \langle u_1, u_2\rangle \ge 1 \implies \langle u_1, u_2\rangle = 1 \implies u_1 = u_2.\]
This proves that $u$ must be a generator of $P$. By symmetry,  $v$ must also be a generator of $Q$. Therefore, since any solution pair is a Nash pair we have that if $(u,v)$ is a solution pair of $\SV(A,P,Q)$ then $u=g_i$ and $v=h_j$ for some $i,j$, and $\lambda^*=\min_{i,j}\langle g_i, A h_j\rangle=\min_{i,j}(G^\top A H)_{i,j}$.
\end{proof}

By specializing Proposition\,\ref{pro:nonnegative_case_SVcones} to the problems $\MA(P,Q)$ and $\PSV(A)$, we obtain the following two corollaries, which were already known in the literature.

\begin{corollary}\cite[Theorem 2]{orlitzky2020maximal}  \label{cor:simple_case_angle_dual_cone}
Let $P$ and $Q$ be polyhedral cones generated by $G$ and $H$, respectively, as in~\eqref{polyhedral}. Assume that $Q\subseteq P^*$. Then, the maximal angle between $P$ and $Q$ are achieved by a couple of generators in $P$ and $Q$, that is, the maximal angle between $P$ and $Q$ is $\arccos \left( \min_{i,j}\langle g_i,h_j\rangle \right) \in [0,\pi/2]$.
\end{corollary}
\begin{proof}
Since $P=G(\mathbb R^p_+)$, $z\in P^*$ if and only if $G^\top z\geq 0$. Then,
$$G^\top H\geq 0\;\Leftrightarrow\; \left(G^\top H y\geq 0,\,\forall y\in\R^q_+\right)\;\Leftrightarrow\;\left(\forall v,\, v\in Q\;\Rightarrow\; G^\top v\geq 0\right)\Leftrightarrow Q\subseteq P^*.$$
The result follows from Proposition~\ref{pro:nonnegative_case_SVcones} by recalling that $\MA(P,Q)=\SV(I,P,Q)$. 
\end{proof}

\begin{corollary}\cite[Corollary 1]{seeger2023PSV} 
Let $A=(a_{i,j})\in\R^{m\times n}$ be nonnegative. Then, the least Pareto singular value of $A$ is $\min_{i,j}a_{i,j}$, and it is achieved by a couple of canonical vectors in $\R^m_+$ and $\R^n_+$.
\end{corollary}
\begin{proof}
    Since $\PSV(A) = \SV(A,I,I)$, then by Proposition \ref{pro:nonnegative_case_SVcones} the optimal value is the minimum entry of $G^\top A H = A$. 
\end{proof}

\subsection{Identifying saddle points}

In general, the problem $\SV(A,P,Q)$ has many critical pairs. For instance, in \cite{seeger2023SV} it is shown that for all $m,n\geq 1$ there exists $A\in\R^{m\times n}$ such that $\PSV(A)$ has $(2^m-1)(2^n-1)$ critical pairs. In some cases, many of the critical pairs are saddle points. In Section\,\ref{sec:exactalgo}, we will see a method to solve
$\PSV(A,P,Q)$ by confronting all critical pairs of the problem.
Hence, 
it will be advantageous to recognize in advance which critical pairs are
saddle points so they can be omitted (see Algorithm\,\ref{alg:active_set}).

\begin{theorem}
 \label{the:saddle}   Let  $(u,v)$ be a critical pair of $\MA(P,  Q)$ and let $u\in F_u$, $v\in F_v$ where $F_u$, $F_v$ are the faces of $P, Q$ with the least dimension possible containing $u,v$. If $\dim(F_u) + \dim(F_v)> n$ and $v\ne \pm u$, then $(u,v)$ is a saddle point.  
\end{theorem}
\begin{proof}
It is enough to construct twice continuously differentiable curves $\theta\mapsto u(\theta)$ and $\varphi\mapsto v(\varphi)$ (both around $0$) such that $(u(0),v(0))=(u,v)$, $(u(\theta),v(\varphi))$ is feasible for $\MA(P,Q)$, and checking that the determinant of the Hessian of $(\theta,\varphi)\mapsto\langle u(\theta),v(\varphi)\rangle$ at $(0,0)$ is negative.

To do that, observe that $n>1$, and $u,v$ belong to the relative interior part of $F_u$, $F_v$. Call now $S_u$ and $S_v$ the span of $F_u$, $F_v$. From the hypothesis,     \[
    \dim(S_u) + \dim(S_v) = \dim(F_u) + \dim(F_v) > n \implies \dim(S_u\cap S_v) > 0.
    \]
The fact $v\ne \pm u$ ensures that $v\notin S_u$ and $u\notin S_v$. Let $z\in S_u\cap S_v$. Because of the above discussion, we have that $\{u,v,z\}$ is linearly independent. Let $R_u(\theta)$ be an orthogonal matrix that fixes all vectors in the subspace orthogonal to $V_u:= {\rm Span}\{u,z\}$ and rotates the vectors of the plane $V_u$ by an angle $\theta$. We take $u(\theta):=R_u(\theta)u$ and observe that $u(0)=u$, and $\theta\mapsto u(\theta)\in F_u\cap S_m$ is twice continuously differentiable for small $|\theta|$. Similarly, let $v(\varphi):=R_v(\varphi)v$ be constructed in the same way over the subspace $V_v:= {\rm Span}\{v,z\}$. We have that $v(0)=v$, and $\varphi\mapsto v(\varphi)\in F_v\cap S_n$ is twice continuously differentiable for small $|\varphi|$. Thus, $(u(\theta),v(\varphi))$ is feasible for $\MA(P,Q)$ around $(0,0)$, and the Hessian of $(\theta,\varphi)\mapsto\langle u(\theta),v(\varphi)\rangle$ at $(0,0)$ is
\[
\begin{pmatrix}
- \langle u,v\rangle &   1\\
 1& - \langle u,v\rangle
\end{pmatrix},
\]
which has determinant $\langle u,v\rangle^2 -1< 0$.  This proves that $(u,v)$ is always a saddle point. We omit the details of the above computation since they are elementary but quite technical.
\end{proof}

Through the reduction from  $\SV(A,P,Q)$ to $\MA(P,  Q)$ we can also identify saddle points for the general problem of conical singular values. 

\begin{corollary}
 \label{cor:saddle_ConeSV}   Let  $(u,v)$ be a critical pair of $\SV(A,P,Q)$ and let $u\in F_u$, $v\in F_v$ where $F_u$, $F_v$ are the faces of $P, Q$ with the least dimension possible.
 Let $r$ be the multiplicity of the singular value $\|A\|$ of the matrix $A$.   
 If $\dim(F_u) + \dim(F_v)> m+n-r$ and $\langle u, Av\rangle\ne \pm \|A\|$, then $(u,v)$ is a saddle point.  
\end{corollary}
\begin{proof}
Notice that the statement is symmetric in $P,Q$ just by changing $A$ into $A^\top$. As a consequence, we can suppose $m\ge n$. 
From the proof of Theorem \ref{th:reduction_SVcones_anglescones}, we can reduce the conic singular value problem to the conic angles problem through the decomposition $A = \|A\| U^\top V$ of Lemma \ref{lem:orthogonal_decomposition} as $\SV(A,P,Q)=\Vert A\Vert\MA(\wt P,\wt Q)$, 
where $\wt P=UP$ and $\wt Q=VQ$ are cones in $\f R^{m+s}$ and $(u,v)$ is a critical pair of $\SV(A,P,Q)$ if and only if $(x,y) = (Uu,Vv)$ is a critical pair of $\MA(\wt P,\wt Q)$.  
Here $s$ is equal to  $\rank(A^\top A-\|A\|^2 I)$, that is equal to  $n-r$.   
Since $\wt P$ is an immersion of $P$ into $\f R^{m+n-r}$, and  $x\in F_x$ where $F_x$ is the face of $\wt P$ with the least dimension possible containing $x$, then $\dim(F_x) = \dim (F_u)$. Analogously, $\dim(F_y) = \dim (F_v)$. Moreover,
\[
\langle u, Av\rangle= \pm \|A\| \iff \langle U u, V v\rangle = \pm 1 \iff x = Uu=\pm Vv= \pm y.
\]
Theorem~\ref{the:saddle} allows us conclude that if  $\dim(F_u) + \dim(F_v)> m+n-r$ then $(x,y)$ is a saddle point for  $\MA(\wt P,\wt Q)$ and as a consequence,  $(u,v)$ is a saddle point for $\SV(A,P,Q)$.  
\end{proof}

{Observe that $-\Vert A\Vert\leq \langle u,Av\rangle\leq\Vert A\Vert$ for all feasible pair $(u,v)$ of $\SV(A,P,Q)$.} Thus, the last result reduces the number and dimension of faces we need to test to find the optimal solution, but it requires that $\langle u, Av\rangle =  \pm  \|A\|$ is not attained for an optimal pair  $(u,v)$ of $\SV(A,P,Q)$. 
 The case $\langle u, Av\rangle =    \|A\|$ is easy to check, since it is equivalent to say that both $P,Q$ have only one generator and $G^\top AH$ is a $1\times 1$ nonnegative matrix. 

The case $\langle u, Av\rangle =  -  \|A\|$ is also identifiable in polynomial time; see 
in Section~\ref{subs:preprocessing} 
where we report the algorithm that checks whether such a solution exists. 

\subsection{Cases when an optimal vector is a generator}\label{sec:simple_cases_generators}

Another consequence of Theorem \ref{the:saddle} is that in low dimensions we know that all solution pairs $(u^*,v^*)$ of $\MA(P, Q)$ must contain at least a generator of the respective cone.
 
\begin{corollary}
     \label{cor:local_minimum_angles_have_generators_in_dim3} Let  $(u,v)$ be a local minimum of $\MA(P, Q)$ with $u \neq -v$ in dimension $n\le 3$. Then $u$ or $v$ is a generator of its respective cone. 
\end{corollary}
\begin{proof}
  If $u=v$ then $P,Q$ are both generated by $u$ since $(u,v)$ is a local minimum, so $u$ is a ray of both cones. We can thus suppose that $u\ne v$.  
  Since $(u,v)$ is not a saddle, then necessarily $\dim(F_u) + \dim(F_v)\le n\le 3$  by Theorem \ref{the:saddle},  where $F_u$, $F_v$ are the faces of $P, Q$ with the least dimension possible containing $u,v$. Since the faces have positive dimensions, one of the two must necessarily have dimension 1, and the respective vector $u$ or $v$ will thus be a generator for the respective cone.
\end{proof}

When the dimension is larger than 3, the result does not hold anymore. Here is a counterexample: take  the cones $P,Q\cu \f R^4$, where
$P$ is generated by $(1 ,-1, 0, 0)^\top$, $(1,1,0,0)^\top$, and $Q$ is generated by $(-1,0,1,-1)^\top$, $(-1,0,1,1)^\top$. The only solution pair is $u=(1,0,0,0)^\top$ and $v=\frac{1}{\sqrt{2}}(-1,0,1,0)^\top$
that are both not rays of the respective cones. In this case the faces where they lie are both of dimension $2$ and $2+2$ is not larger than $4$, so Theorem \ref{the:saddle} does not apply to the pair $(u,v)$. In fact, it is not necessarily a saddle point, but instead it is the global optimal solution.   \\

Corollary~\ref{cor:local_minimum_angles_have_generators_in_dim3} states that any solution $(u^*,v^*)$ of $\MA(P, Q)$ in dimension $3$ or less contains at least a generator. As a consequence, the minimum over the optimal values of $\MA(\f R_+g_i,Q)$ and $\MA(P,\f R_+ h_j)$ over all $i,j$ corresponds to the maximal angle of $\MA(P, Q)$.  The problem thus reduces to  solve 
\begin{equation}\label{eq:angles_cones_convex}
\MA(\f R_+z,Q):\;\;\min_{v} \langle z,v\rangle\quad\text{s.t.}\quad v\in Q\cap S_n,
\end{equation}
where $z$ is an unit vector, once for each $z=g_i$ generator of $P$, and then the specular problem $\MA(P,\f R_+z)$ for each $z=h_j$ generator of $Q$. 

More in general, we will need a way to solve the problem (see Subsection\,\ref{sec:AO})
\begin{equation}\label{eq:conic_sv_convex}
\SV(A,\f R_+z,Q):\;\;\min_{v} \langle z,Av\rangle\quad\text{s.t.}\quad v\in Q\cap S_n,
\end{equation}
for any unit norm vector $z$. 
Problem~\eqref{eq:conic_sv_convex}, and consequently also \eqref{eq:angles_cones_convex}, can be solved easily by well-known methods. In fact, when $\langle z,Av\rangle\ge 0$ for all $v\in Q$, then the solution is a generator of $Q$ as proven by 
Proposition~\ref{pro:nonnegative_case_SVcones} applied to the cones $( \f R_+z,Q)$. Otherwise, it can be reduced to the problem of projecting $-A^\top z$ onto the polyhedral convex cone $Q$, {see \cite{bauschke-projectingonto2018}, which is a convex quadratic problem.

\section{Algorithms for cone-constrained singular values} \label{sec:algorithms}

In this section, we propose 4 algorithms to tackle $\SV(A,P,Q)$ with $P,Q$ polyhedral cones generated by $G,H$ respectively: two exact algorithms that come with global optimality guarantees but may run for a long time (Section~\ref{sec:exactalgo}), and two heuristics that do not provide global optimality guarantees but run fast (Section~\ref{sec:heuristics}). 
Before doing so, we explain in Section~\ref{subs:preprocessing} how to leverage the results of the previous section, allowing us to identify cases that we know are solvable in polynomial time.

\subsection{Preprocessing: Checking the simple cases} \label{subs:preprocessing}

Let  us denote $(u^*,v^*)$ a solution pair of $\SV(A,P,Q)$ and $\lambda^*$ the associated optimal value.

\paragraph{Case 1:  $\lambda^* \geq 0$.}

If $\lambda^*$ is nonnegative, then  $ 0\le \lambda^*\le \min_{i,j}g_i^\top Ah_j$  because the columns of $G$ and $H$ are unit vectors. As a consequence $G^\top AH\ge 0$ and from Proposition \ref{pro:nonnegative_case_SVcones} we can conclude that $\lambda^*\ge 0$ if and only if  $G^\top AH\ge 0$. The same result shows that  $\lambda^* = \min_{i,j}g_i^\top Ah_j$, which can be easily detected in a preprocessing step. For this reason, from now on, we assume that $\lambda^*<0$.

\paragraph{Case 2: $\lambda^* =  - \|A\|$.}  

The case $\lambda^* =  - \|A\|$ can be checked using the singular value decomposition and solving convex quadratic programs.  
In fact, it is equivalent to say that $v^*$ is a right singular vector for $-A$ associated with the singular value $\|A\|$ and the left singular vector $u^*$. We take $U,V$ as the basis for the left and right singular vectors of $A$ relative to $\|A\|$, respectively, and realize that there must exist a vector $w^*\ne 0$ for which $v^* = Vw^*$ and $u^* = -Uw^*$. 
Since $u^*\in P$ and $v^*\in Q$, then  $u^* = G x^*$ and $v^* = Hy^*$ with $x^*,y^*$ nonnegative coefficients, that we can join in $ z^*:= \begin{pmatrix}
         y^{*\top}  &  x^{*\top} 
    \end{pmatrix}^\top\ge 0$.  Then for any  $W$ with orthonormal columns spanning the same subspace as $\begin{pmatrix}
     V\\ -U  \end{pmatrix}$, there exists a nonzero $w$ such that 
  $\begin{pmatrix}
     v^*\\ u^*  \end{pmatrix} = 
     W w = \begin{pmatrix}
       H& 0\\0 & G  \end{pmatrix} z^* \ne 0.
     $ So,
\[
\exists w\ne 0\ :\  \begin{pmatrix}
       H& 0\\0 & G  \end{pmatrix} z^* = Ww \iff  (I-WW^\top )\begin{pmatrix}
       H& 0\\0 & G  \end{pmatrix} z^* = 0, \quad \begin{pmatrix}
       H& 0\\0 & G  \end{pmatrix} z^* \ne 0
\]
We thus need to check if $0$ is the optimal value for the optimization problem
       \begin{equation}\label{eq:extreme_case}
              \min_z \left\|(I-WW^\top )\begin{pmatrix}
       H& 0\\0 & G  \end{pmatrix} z\right\| \, : \, z\ge 0,\, \begin{pmatrix}
       H& 0\\0 & G  \end{pmatrix} z\ne 0.
       \end{equation}
Notice that if the problem \begin{equation}\label{eq:extreme_case_rand+}
              \min_z \left\|(I-WW^\top )\begin{pmatrix}
       H& 0\\0 & G  \end{pmatrix} z\right\| \, : \, z\ge 0,\, p^\top \begin{pmatrix}
       H& 0\\0 & G  \end{pmatrix} z=\kappa, 
       \end{equation}
 has optimal value equal to zero for some vector $p$, and either for $\kappa=1$ or $\kappa=-1$, then also \eqref{eq:extreme_case} has the same  optimal solution $z^*$ with the same zero optimal value.
The converse, though, is not true, since the optimal $z^*$ solving \eqref{eq:extreme_case} might satisfy  $p^\top \begin{pmatrix}
       H& 0\\0 & G  \end{pmatrix} z^*=0$, but if $p$ is drawn randomly from any continuous distribution, then the probability for it to happen is zero.  On the other hand, if  $z^*$ solves \eqref{eq:extreme_case}, then there exists an index $i$ and a positive constant $\mu$  such that
 $e_i^\top \begin{pmatrix}
       H& 0\\0 & G  \end{pmatrix} \mu z^*=\pm 1$, so   \eqref{eq:extreme_case_rand+} has optimal value $0$ for $p=e_i$. 
       
   As a consequence, we can transform the problem to check if $\lambda^* =  - \|A\|$ into a set of easily solvable least squares problems with linear constraints. We propose two approaches to do so, a fast but randomized formulation, or a slower but deterministic one:

\begin{itemize}
    \item Given a randomly chosen vector $p$ drawn from any continuous distribution, and given $z^*$ solving \eqref{eq:extreme_case},  $p^\top \begin{pmatrix}
       H& 0\\0 & G  \end{pmatrix} z^*=0$ with probability zero. In particular, if $0$ is the optimal value of \eqref{eq:extreme_case}, then $0$ is also the optimal value of  \eqref{eq:extreme_case_rand+} (for either $\kappa=1$ or $\kappa=-1$) with probability 1.  If $0$ is not the optimal value of \eqref{eq:extreme_case}, then $0$ is also not the optimal value of either  \eqref{eq:extreme_case_rand+}. 
       
       \item  $0$ is the optimal value of \eqref{eq:extreme_case} if and only if  $0$ is also the optimal value of  \eqref{eq:extreme_case_rand+} for $p=e_i$ and some index $i$. As a consequence, 
we can look for a zero optimal value of  \eqref{eq:extreme_case_rand+}   for all possible $p=e_i$. This is computationally more expensive than the randomized approach above since we need to solve $n$ problems in the worst case, but it has the advantage of being deterministic.
 \end{itemize}

\subsection{Exact algorithms} \label{sec:exactalgo}

In this section, we propose two exact algorithms to tackle $\SV(A,P,Q)$: a brute-force active-set that enumerates all possible supports of the solutions $(u^*,v^*)$ (Section~\ref{sec:bruteforce}), and a formulation that can be solved via the non-convex quadratic solver of Gurobi (Section~\ref{sec:gurobi}).

\subsubsection{Brute-force active-set method for polyhedral cones} \label{sec:bruteforce}

In \cite{seeger-sossa2016angle1}, the authors used a brute-force active-set method to find exactly all the Pareto singular values in $\PSV(A)$. A similar method has been used to find all the critical angles between two cones in \cite{orlitzky2020maximal}. It is thus not surprising that the same reasoning can be used to solve exactly $\SV(A,P,Q)$, that is
\begin{equation}\label{eq:SVcones_generators}
    \min_{\scriptsize \begin{array}{l}u\in P,\,\Vert u\Vert=1,\\v\in Q,\,\Vert v\Vert=1\end{array}}\langle u,Av\rangle =
 \min_{\scriptsize \begin{array}{l}x\ge 0,\,\Vert Gx\Vert=1,\\y\ge 0,\,\Vert Hy\Vert=1\end{array}}\langle Gx,AHy\rangle.
\end{equation}
 The main idea is that, for an optimal pair $(x^*,y^*)$ we can restrict $x^*,y^*,G$ and $H$ to the non-active sets $\mathcal I:=\{i:x^*_i>0\}$ and $\mathcal J:=\{i:y^*_i>0\}$ obtaining $\ol G = G_{:,\mathcal{I}}, \ol H = H_{:,\mathcal{J}}$ and discover that the restricted optimal $\ol x^* = x^*_{\mathcal{I}}$ and $\ol y^* = y^*_{\mathcal J}$ solve 
 \[
 \min_{\scriptsize \begin{array}{l}\ol x> 0,\,\Vert \ol G\ol x\Vert=1,\\\ol y> 0,\,\Vert \ol H \ol 
 y\Vert=1\end{array}} \langle \ol G\ol x,A\ol H\ol y\rangle.
 \]
The pair $(\ol x^*,\ol y^*)$ will also be a local minimum for the same problem when we remove the positivity constraints $\ol x,\ol y> 0$, and can be computed using
  classical Lagrange multipliers
, see Theorem~\ref{th:critical_points_SVcones} below. As a consequence, one could test every index set $\mathcal I$ and $\mathcal J$ and solve the reduced problems to find $(x^*,y^*)$.
The following result, whose proof is similar to the one in \cite{seeger-sossa2016angle1}, shows some necessary conditions for a pair $(x^*,y^*)$ to be optimal. {We denote by $M^+$  the pseudoinverse of a full column rank matrix $M\in\R^{m\times n}$, that is, $M^+:=(M^\top M)^{-1}M^\top \in\R^{n\times m}$.}

\begin{theorem}\label{th:critical_points_SVcones}
  Let $A\in\R^{m\times n}$, and let $P=G\R^p_+$ and $Q=H\R^q_+$ be polyhedral cones as in \eqref{polyhedral}. Let $( u^*, v^*)\in \f R^m\times \f R^n$  be an optimal solution of problem $\SV(A,P,Q)$, 
with $\lambda^* := (u^*)^\top Av^*\ne \pm\|A\|$ and $ \lambda^*\le 0$. 
Then there exist $x^*,y^*\ge 0$ such that $u^* = Gx^*$, $v^*=Hy^*$ for which $\emptyset\ne\mathcal I:=\{i:x^*_i>0\} \cu [p]$, $\emptyset\ne\mathcal J:=\{i:y^*_i>0\} \cu [q]$ and that the following four properties hold: 
\begin{itemize}
    \item Property 1.   $|\mathcal I| + |\mathcal J|\le n + m-r$, where $r$ is the multiplicity of the singular value $\|A\|$ of $A$.
    
    \item Property 2. $\ol G:= G_{:,\mathcal I}$ and $\ol H:= H_{:,\mathcal J}$ are full column rank.
    
    \item Property 3. if $\ol x^* := x^*_{\mathcal I}$,  $\ol y^* := y^*_{\mathcal J}$,  $\ol z^*:= \begin{pmatrix}
        \ol y^{*\top}  & \ol x^{*\top} 
    \end{pmatrix}^\top$ and $M := \begin{pmatrix}
       0& \ol H^+ A^\top\ol G\\\ol G^+ A\ol H& 0  \end{pmatrix}$, then $\lambda^*$ is the least eigenvalue of $M$ and $\ol z^*$ belongs to its eigenspace. 
       
    \item Property 4. If $\wt G:= G_{:,\mathcal I^c}$ and $\wt H:= H_{:,\mathcal J^c}$ then 
    $\begin{pmatrix}
       0& \wt H^\top A^\top\ol G\\\wt G^\top A\ol H& 0  \end{pmatrix} \ol z^*
       $$-$$\lambda^* \begin{pmatrix}
        \wt H^\top \ol H & 0\\ 0& \wt G^\top \ol G   \end{pmatrix}\ol z^*$$\ge$$0
        $.
\end{itemize}
\end{theorem}
\begin{proof}
Let us rewrite  $\SV(A,P,Q)$ in terms of the generators $G,H$ of the cones $P,Q$, respectively, as in \eqref{eq:SVcones_generators}. 
Necessary conditions for stationarity of \eqref{eq:SVcones_generators} are given by the following KKT conditions 
  \begin{equation}\label{eq:KKT_SVcones}
\begin{cases}
    0\le x\perp G^\top AH y - \lambda G^\top Gx \ge 0, \\
    0 \le y\perp H^\top A^\top G x - \lambda H^\top Hy \ge 0, \\
    \|Gx\| = \|Hy\| = 1. 
\end{cases}
\end{equation}
If now $( u^*, v^*)\in P\times Q$ is an optimal solution of problem $\SV(A,P,Q)$, then by Carathéodory's theorem, $u^*$ and $v^*$ are generated by some linearly independent subset of generators of their respective cones. In other words, there exist $x^*,y^*\ge 0$, such that $u^* = Gx^*$, $v^*=Hy^*$ and if $\mathcal I:=\{i:x^*_i>0\} \cu [p]$ and $\mathcal J:=\{i:y^*_i>0\} \cu [q]$  then $\ol G:= G_{:,\mathcal I}$ and $\ol H:= H_{:,\mathcal J}$ are full column rank. Property 2 is thus satisfied.
 Considering $\ol x^*, \ol y^*, \wt G, \wt H$ from the thesis, we can rewrite \eqref{eq:KKT_SVcones} as \begin{equation}\label{eq:restricted_KKT}
 \begin{cases}
 0 < \ol x^*,\quad  \ol G^\top A\ol H \ol y^* - \lambda \ol G^\top \ol G\ol x^* = 0,\\
 0 < \ol y^*,\quad  \ol H^\top A^\top \ol G \ol x^* - \lambda \ol H^\top \ol H\ol y^* = 0,\\
    \wt G^\top A\ol H \ol y^* - \lambda \wt G^\top \ol G\ol x^* \ge 0, \\
    \wt  H^\top A^\top\ol G \ol x^* - \lambda \wt H^\top \ol H\ol y^* \ge 0, \\
    \|\ol G\ol x^*\| = \|\ol H\ol y^*\| = 1.
\end{cases}
\end{equation}
Notice that $\lambda^*=  \langle u^*,Av^*\rangle = \langle \ol G\ol x^*,A\ol H\ol y^*\rangle= \lambda \|\ol G\ol x^*\|^2 = \lambda$, so Property 4 corresponds to the third and fourth inequalities of \eqref{eq:restricted_KKT}.  Since the $\ol G, \ol H$ have full column rank, then $\ol G^\top \ol G, \ol H^\top \ol H$ are invertible and thus  $\ol G^+A\ol H \ol y^* - \lambda^*\ol x^* = 0$, $\ol H^+ A^\top \ol G \ol x^* - \lambda^* \ol y^* = 0$ , that is, $M\ol z^* = \lambda^* \ol z^*$.
Notice that $\ol x^*, \ol y^*$ are strictly positive vectors, and they  solve the minimization problem
\[    \min_{\scriptsize \begin{array}{l}\ol x> 0,\,\Vert \ol G\ol x\Vert=1,\\\ol y> 0,\,\Vert \ol H\ol 
 y\Vert=1\end{array}}\langle \ol G\ol x,A\ol H\ol y\rangle\;=\; \lambda^*.  
\]
As a consequence, $(\ol x^*,\ol y^*)$ is a local minimum for the simpler problem $\min_{ \Vert \ol G\ol x\Vert=\Vert \ol H\ol y\Vert=1}\langle \ol G\ol x,A\ol H\ol y\rangle$. 
The vector $\ol z^*$ must thus satisfy the necessary second order conditions for the local minima, that are 
\begin{equation}
    \label{second_order_KKT}
     \begin{cases}
2\langle \ol G w_x,A\ol Hw_y\rangle\ge \lambda^* (\|\ol Gw_x\|^2 + \|\ol Hw_y\|^2) & \forall (w_y,w_x)\in \mathcal Z(\ol z^*), \\
 (M-\lambda^* I) \ol z^* = 0,\\
\Vert \ol G\ol x^*\Vert=\Vert \ol H\ol y^*\Vert=1,
\end{cases}
\end{equation}
where $ \mathcal Z(\ol z^*):= \left\{ (w_y,w_x)\in \mathbb R^{|\mathcal J|}\times \mathbb R^{|\mathcal I|} : \langle \ol G w_x,\ol G\ol x^*\rangle = \langle \ol H w_y,\ol H\ol y^*\rangle  =0 \right\}$.  Since $\ol G$ and $\ol H$ are full column rank, then we have the following similitude relations
\[
M = \begin{pmatrix}
       0& \ol H^+ A^\top\ol G\\\ol G^+ A\ol H& 0  \end{pmatrix} 
       = \begin{pmatrix}
        \ol H^\top \ol H & 0\\ 0& \ol G^\top \ol G   \end{pmatrix} ^{-1}
       \begin{pmatrix}
       0& \ol H^\top A^\top\ol G\\\ol G^\top A\ol H& 0  \end{pmatrix} =: N^{-1}B \sim N^{-1/2}BN^{-1/2}
\]
and the last matrix is in particular real and symmetric, so $M$ is diagonalizable and it has a real spectrum. Notice that since $B, N$ are symmetric, then $NM = B=B^\top = M^\top N$. The first condition of \eqref{second_order_KKT} can now be rewritten as
\begin{equation}
    \label{second_order_KKT2}
\langle w,Bw\rangle\ge \lambda^* \langle w,Nw\rangle  \qquad \forall w\in \mathcal Z(\ol z^*)
\end{equation}
 where $ \mathcal Z(\ol z^*):= \left\{ w\in \mathbb R^{|\mathcal J|+|\mathcal I|} : \langle   w,N \ol z^*\rangle = \langle \wt w,N \ol z^*\rangle  =0 \right\}$ and  $ w:= \begin{pmatrix}
         w_y^\top  & w_ x^\top 
    \end{pmatrix}^\top$, $ \wt w:= \begin{pmatrix}
         w_y^\top  & -w_ x^\top 
    \end{pmatrix}^\top$.
Suppose now $\mu$ is any real eigenvalue of $M$ such that $|\mu|\ne |\lambda^*|$, with an eigenvector $s:= \begin{pmatrix}
        s_y^\top  & s_ x^\top 
    \end{pmatrix}^\top$ normalized as $ \langle  s,N s\rangle  = 1$.  Notice that $M\wt s = -\mu \wt s$ and that from $M\ol z^* = \lambda^* \ol z^*$, 
    \begin{align*}
        \mu \langle  s,N \ol z^*\rangle
= \langle  Ms,N \ol z^*\rangle
 = \langle  N s,M \ol z^*\rangle 
 = \lambda^*\langle  N s, \ol z^*\rangle \ &
\implies \langle  Ns, \ol z^*\rangle = 0,\\
-  \mu \langle  \wt s,N \ol z^*\rangle
= \langle  M\wt s,N \ol z^*\rangle
 = \langle  N \wt s,M \ol z^*\rangle 
 = \lambda^*\langle  N \wt s, \ol z^*\rangle \ &
\implies \langle  N\wt s, \ol z^*\rangle = 0,
    \end{align*}
which in turn implies that $s\in \mathcal Z(\ol z^*)$. 
Finally, from  \eqref{second_order_KKT2}, 
\[
               \mu =     \mu\langle  s,Ns\rangle =  \langle  Ms,N s\rangle= \langle  s,B s\rangle\ge \lambda^*\langle  s,N s\rangle = \lambda^*,
\]
so  $\mu \ge \lambda^*$ and repeating the same reasoning with $-\mu$, one gets $-\mu\ge \lambda^*$. Since by hypothesis  $\lambda^*\le 0$, then $\lambda^*$ is the least eigenvalue of $M$ and Property 3 is satisfied. 

Since $\ol G$  is full rank, then  $u^* =  \ol G \ol x^*$ is in the interior part of a face of at least dimension $|\mathcal I|$ in $P$ and analogously $v^* =  \ol H \ol y^*$ is in the interior part of a face of at least dimension $|\mathcal J|$ in $Q$. Since $\lambda^*\ne \pm \|A\|$, then by Corollary \ref{cor:saddle_ConeSV} we can impose $|\mathcal I| + |\mathcal J|\le n + m-r$, otherwise the point $(u^*,v^*)$ will be a saddle point, and thus not a global minimum. Property 1 is thus also proved.
\end{proof}

Properties 1-4 in Theorem \ref{th:critical_points_SVcones} are necessary for $(x^*,y^*)$ to be an optimal solution, but they are not in general sufficient. In fact, these properties are common to many critical pairs of $\SV(A,P,Q)$, and in general they can only guarantee that the computed value $\lambda$ is an upper bound of the optimal $\lambda^*$. This shows the need to test every pair of subsets of indices $(\mathcal I,\mathcal J)$ to be sure to find the optimal solution.\\

A problem that has never been addressed before is how to check whether there exists a nonnegative and nonzero vector $\ol z^*$ belonging to the $\lambda$-eigenspace of $M$, that is, how to check whether 
Property~3 of Theorem~\ref{th:critical_points_SVcones} is satisfied when the eigenspace has dimension larger than one. 
 In fact, if the eigenspace has dimension one, then it is enough to take a single eigenvector $w$ of $M$
 and check if $w$ is either nonnegative or nonpositive, and in that case $\ol z^* = w$ or $\ol z^* = -w$, respectively.
 
 When instead the eigenspace has dimension larger than 1, then we need to solve a least square problem with linear constraints as described by the following proposition.

\begin{proposition}
\label{pro:sufficient_conditions_feasibility}
Suppose that $\emptyset\ne \mathcal I \cu[p]$, $\emptyset\ne \mathcal J \cu[q]$ with $|\mathcal I|\le |\mathcal J|$. Let $A\in \f R^{m\times n}$, $G\in \f R^{m\times p}$, $H\in \f R^{n\times q}$ such that the columns of $G$ and $H$ have unit norm and $\ol G:= G_{:,\mathcal I}$ and $\ol H:= H_{:,\mathcal J}$ are full column rank. 
Let  $A_x:=\ol H^+ A^\top\ol G$,  $A_y:=\ol G^+ A\ol H$ and let $U_G$ and $U_H$ be {matrices whose columns are} orthonormal basis for the images
of  $\ol G$ and $\ol H$ respectively. 
Suppose moreover that $U$ is a basis of the $\rho(A_yA_x)$-eigenspace of $A_yA_x$ and let $V$ be an orthonormal basis for the image of  $ \begin{pmatrix}
      A_xU/\lambda\\ U  \end{pmatrix}$. Then
\begin{enumerate}
    \item  the spectral radius of $A_yA_x$ is $\rho(A_yA_x) = \|U_G^\top AU_H\|^2$,
    \item  $\lambda:= - \sqrt{\rho(A_yA_x)}$ is the least eigenvalue of $ M:= \begin{pmatrix}
       0& A_x\\A_y& 0  \end{pmatrix}$,
       \item  $
     \min_{ z\ge 0,\,\, e^\top z = 1} \left\| (VV^\top - I)  z\right\| = 0
     \iff 
       \exists\, \ol z\in \f R_+^{q+p}, \, \ol z\ne 0: 
 M
       \ol z = \lambda \ol z$,
       
       \item if $\ol z\ne 0$ solves $VV^\top \ol z = \ol z\ge 0$, then given the decomposition $\ol z^\top = [\ol y^\top \,\,\ol x^\top]$ with $\ol x\in \f R^{|\mathcal I|}$, $\ol y\in \f R^{|\mathcal J|}$, and 
 $u = \ol G \ol x/\|\ol G\ol x\|$, $v=\ol H\ol y/\|\ol H\ol y\|$,  we have that $(u,v)$ is a feasible point for problem $\SV(A,P,Q)$ with $\lambda = u^\top Av$. 
 \end{enumerate}
\end{proposition}
 \begin{proof}
 First of all, $A_yA_x$ is similar to $B B^\top $ where 
 \[
 B:= (\ol G^\top\ol G)^{-1/2}\ol G^\top A\ol H (\ol H^\top\ol H)^{-1/2},
 \]
 so $A_yA_x$ is diagonalizable and all its eigenvalues are nonnegative. In particular, $ \rho(A_yA_x) $ is an eigenvalue of $A_yA_x$. Moreover, $\ol H (\ol H^\top\ol H)^{-1/2} = U_HV_H$ where $V_H$ is a square orthogonal matrix, and similarly $\ol G (\ol G^\top\ol G)^{-1/2} = U_GV_G$, so  
      $\rho(A_yA_x) =\| B\|^2 = \|U_G^\top AU_H\|^2$.
If $ M w = \mu w$, where $w = [w_y^\top \,\,w_x^\top]^\top $ and $\mu$ are nonzero, then $A_yA_x w_x = \mu A_y w_y = \mu^2 w_x$  and $w_x$ is nonzero, so $\mu^2$ is an eigenvalue of $A_yA_x$. Vice versa, if $A_yA_x w_x = \mu^2 w_x$ with $w_x$ and $\mu$ nonzero, then we can call
$w_y := A_xw_x/\mu $ and obtain that $A_yw_y = \mu w_x$, so $w = [w_y^\top \,\,w_x^\top]^\top \implies M w= \mu w$. As a consequence,
\[
\{\pm \sqrt{\lambda}:\lambda\in\Lambda(A_yA_x)\}\;=\;\Lambda\left(M\right)\setminus\{0\},  
\]
where $\Lambda(C)$ denotes the spectrum of a square matrix $C$. In particular, $-\sqrt{\rho(A_yA_x)}$ is the least eigenvalue of $M$. 
     
Since $VV^\top$ is the orthogonal projection on the space  $E:= \text{Span}(V)$ that coincides with the space spanned by $ \begin{pmatrix}
       A_xU/\lambda\\ U  \end{pmatrix}$, 
       checking whether $\min_{ z\ge 0,\,\, e^\top z = 1} \left\| (VV^\top - I)  z\right\|$ is equal to zero is equivalent to checking whether there exists a nonzero and nonnegative vector $\ol z = [\ol y^\top \,\,\ol x^\top]^\top $ in $E$, that is,  there exist $\ol x,\ol y\ge 0$ and a nonzero $w$ such that  $\ol x = Uw$ and $\ol y = A_xUw/\lambda = A_x \ol x/\lambda$. Notice that $A_y \ol y = A_yA_x \ol x/\lambda = \lambda \ol x$ since $U$ is  a basis of the $\lambda^2$-eigenspace of $A_yA_x$, so we obtain  $M       \ol z = \lambda \ol z$.
Viceversa, if $ M       \ol z = \lambda \ol z$ and $\ol z = [\ol y^\top \,\,\ol x^\top]^\top \ge 0$ is nonzero, then $A_yA_x \ol x = \lambda A_y \ol y = \lambda^2 \ol x$ so $\ol x$ is a $\lambda^2$ eigenvector of $A_yA_x$ and there exists a nonzero $w$ such that $\ol x = Uw$ and $\ol y = A_x \ol x/\lambda=A_xUw/\lambda$.\\

Given such a nonzero vector $\ol z\ge 0$, notice that both $\ol x,\ol y\ge 0$ are nonzero since $ \ol y = A_x \ol x/\lambda = 0$ if and only if 
$\ol x = A_y\ol y/\lambda =0$. This means in particular that $\|\ol G \ol x\|>0$ since $\ol G$ is full column 
rank. Moreover, one can verify that $\|\ol G \ol x\|^2=\| \ol H\ol y\|^2$, so
\begin{align*}
  \langle u,Av\rangle  &= \frac{\langle \ol G\ol x,A\ol H\ol y\rangle  }{ \|\ol G \ol x\| \|\ol H\ol y\|} =
   \frac{\langle \ol G \ol x,\ol G \ol G^+ A\ol H\ol y\rangle }{ \|\ol G \ol x\|^2} =
   \frac{\langle \ol G \ol x,\ol GA_y\ol y\rangle }{ \|\ol G \ol x\|^2} = 
   \lambda \frac{ \langle \ol G \ol x,\ol G \ol x\rangle}{ \|\ol G \ol x\|^2} = \lambda.
\end{align*}
\end{proof}

Notice that $ \min_{ z\ge 0,\,\, e^\top z = 1} \left\| (VV^\top - I)  z\right\| = 0$ is a classical least square problem with linear constraints, that can be solved in polynomial time by classical algorithms.

Using the two results combined, we can now come up with a method to solve $\SV(A,P,Q)$, summarized in 
 Algorithm~\ref{alg:active_set}. The method coincides with Algorithm 3 in \cite{orlitzky2020maximal} when generalized to $\SV(A,P,Q)$, with few differences. 
 
Algorithm \ref{alg:active_set} is intended to be used after the preprocessing described in Section \ref{subs:preprocessing}, ensuring that the problem does not have trivial solutions, and in particular that the least singular value of  $\SV(A,P,Q)$ is neither nonnegative nor equal to $-\|A\|$. The explanation for the steps of the algorithm is as follows:
\begin{itemize}
    \item In Step 1, we take $\lambda = \langle g_i,Ah_j\rangle  = {\min_{k,\ell}{(G^\top AH)_{k,\ell}}}$ as an upper bound for the optimal solution, with feasible solution given by $(u,v) = (g_i,h_j)$.
    \item In Step 2, we collect in $\mc I$ all the pairs of index sets $(\mathcal I,\mathcal J )\cu [p]\times[q]$ of nonzero entries for $(x,y)$ satisfying Property 1 and 2 of Theorem \ref{th:critical_points_SVcones}. We also impose $2<|\mathcal I|+|\mathcal J|$ since all the cases for which $1=|\mathcal I|=|\mathcal J|$ coincide with $(u,v) = (g_i,h_j)$ that have already been considered in Step 1.
    \item Once a pair  $(\mathcal I,\mathcal J )$ is fixed, the smallest eigenvalue of the matrix $M$ squared is computed in Steps 4-6 as the opposite of the spectral radius $\rho$ of $A_xA_y$ or $A_yA_x$, depending on which one of the two matrices has smaller size, as shown by point 2 of Proposition \ref{pro:sufficient_conditions_feasibility}. 
    \item In Steps 6-8, we discard the pair  $(\mathcal I,\mathcal J )$ if $\rho\le \lambda^2$ because by Property 3 of  Theorem \ref{th:critical_points_SVcones}, the optimal solution cannot have the least eigenvalue of the matrix $M$ greater than $-\lambda$. Moreover, it is possible to prove that in this case all feasible solutions $(x,y)$ with sparsity pattern  $(\mathcal I,\mathcal J )$ cannot improve the current upper bound $\lambda$.
    \item If  $\rho> \lambda^2$ then we check if $M$ admits a nonnegative eigenvector $\ol z$ relative to its least eigenvalue $\mu=-\sqrt\rho$. If the eigenspace has dimension 1, it is enough to check in Steps 11-12 if $W$ is nonnegative. Otherwise, in Steps 14-15 we solve the problem in point 3 of Proposition \ref{pro:sufficient_conditions_feasibility}. In case the eigenvector exists, by Proposition \ref{pro:sufficient_conditions_feasibility} the eigenvalue $\mu$ is a better upper bound than $\lambda$ to the optimal solution and we can compute a feasible solution $(u,v)$ such that $\mu = \langle u, Av\rangle$. 
\end{itemize}
By Theorem \ref{th:critical_points_SVcones}, for the pair $(\mathcal I,\mathcal J )$ relative to the optimal solution, the least eigenvalue of $M$ coincides with the optimal $\lambda^*$ of the problem  $\SV(A,P,Q)$  and its eigenspace contains a nonnegative eigenvector $\ol z$. As a consequence, Algorithm \ref{alg:active_set} is bound to output the correct solution.

\begin{algorithm}[h!]
\caption{Brute-Force Active-Set method (BFAS) to solve $\SV(A,P,Q)$} \label{alg:active_set}
\begin{algorithmic}[1]
\REQUIRE Matrix $A \in \R^{m\times n}$,  matrices  $G \in \mathbb{R}^{m \times p}$,  $H \in \mathbb{R}^{n \times q}$ with unit columns generating the cones $P \subseteq \mathbb{R}^{m}$ and $Q \subseteq \mathbb{R}^{n}$. Requires also $G^\top A H\not \ge 0$ and $- \|A\|$ is not the optimal value of $\SV(A,P,Q)$.
\ENSURE An exact solution $\lambda = \min \langle u,Av\rangle$   such that $\|u\|=\|v\| = 1$, $u\in P$, $v\in Q$. 
\medskip 
\State $\lambda = \langle g_i,Ah_j\rangle  = {\min_{k,\ell}{(G^\top AH)_{k,\ell}}}$, $u = g_i$, $v = h_j$, $r = \text{Null}(A^\top A - \|A\|^2 I_n) $.
\State $ \mc I := \{ {(\mathcal I,\mathcal J )\cu [p]\times[q]\,:\,} 2<|\mathcal I|+|\mathcal J|\le m+n-r,\,   \ol G:= G_{:,\mathcal I} \text{ and } \ol H:= H_{:,\mathcal J} \text{ full column rank} \}$
\FOR{$(\mathcal I, \mathcal J)\in  \mc I$,}

 \State $ A_y = \ol G^+ A^\top \ol H$, $A_x = \ol H^+A \ol G$. 
\State $A_\lambda = A_yA_x $, $\wt A_\lambda = A_x$ (or  $A_\lambda = A_xA_y$, $\wt A_\lambda = A_y$ if $|\mathcal I| > |\mathcal J| $).
\IF {$\rho(A_\lambda)\le \lambda^2$}
\State Skip to the next $(\mathcal I, \mathcal J)\in  \mc I$.
\ENDIF
\State Compute the right eigenspace $U$ relative to the eigenvalue $\rho(A_\lambda)$ of $A_\lambda$. 
\State  $\mu = -\sqrt{\rho(A_\lambda)}$,  $W = \begin{pmatrix}
      \wt A_\lambda U/\mu\\ U \end{pmatrix}$.
\IF {$\rho(A_\lambda)$ has multiplicity $1$ and $W$ is nonnegative or nonpositive}
\State  $\lambda = \mu$,  $|W| = [y^\top \,\,x^\top]^\top$ (or $|W| = [x^\top \,\,y^\top]^\top$ if $|\mathcal I| > |\mathcal J| $), $u= \ol Gx/\|\ol G x\|$, $v= \ol Hy/\|\ol Hy\|$.
\ELSIF{ $\rho(A_\lambda)$ has multiplicity $>1$}
\State Compute  an orthonormal basis $V$ for the image of $W$.
\State If $ (VV^\top - I)  z = 0$, $z\ge 0$, $e^\top z = 1$ admits a solution $\ol z$, then
\State $\lambda = \mu$, $\ol  z = [\ol  y^\top \,\,\ol x^\top]^\top$ (or $\ol z = [\ol x^\top \,\,\ol y^\top]^\top $if $|\mathcal I| > |\mathcal J| $), $u= \ol G\ol x/\|\ol G \ol x\|$, $v= \ol H\ol y/\|\ol H\ol y\|$.
\ENDIF

\ENDFOR 
\end{algorithmic}
\end{algorithm}

\subsubsection{Non-convex quadratic solver for polyhedral cones with Gurobi} \label{sec:gurobi} 

For polyhedral cones $P$ and $Q$, the sign of the optimal objective value, $\lambda^*$, of problem $\SV(A,P,Q)$ can be checked; see Section~\ref{sec:simplecases}. 
When $\lambda^* < 0$, the challenging case, $\SV(A,P,Q)$ can be reformulated as follows:  
\begin{equation} \label{eq:SVconesineq}
\lambda^* \; = \; \min_{u \in P, v \in Q} \langle u,Av\rangle
\quad \text{ such that } \quad 
\|u\| \leq 1 \text{ and } \|v\| \leq 1 . 
\end{equation} 
In fact, the norm constraints will be active at optimality since $\langle u,Av\rangle$ is linear in $u$ and $v$ separately, and $\lambda^* < 0$. 
This reformulation~\eqref{eq:SVconesineq} has a convex feasible set over which we optimize a non-convex quadratic objective.  
When $P$ and $Q$ are polyhedral cones (defined either with inequalities or via generators), the global non-convex optimization software Gurobi\footnote{\url{https://www.gurobi.com/solutions/gurobi-optimizer/}} can solve such problems. Let us briefly explain how it works. Gurobi relies on so-called McCormick relaxations~\cite{mccormick1976computability}. 
It introduces an auxiliary variable for each product of variables (the non-convex terms in the objective): 
Let $\Omega = \{(i,j) \ | \ A(i,j) > 0 \}$ and $d = |\Omega|$. Let us denote $i_k$ and $j_k$ the indices of the entry of $A$ corresponding to the $k$-th entry of $\Omega$, that is, 
$\Omega = \{ (i_k,j_k) \ | \ k=1,2,\dots,d\}$. 
Let us also denote the set $\mathcal{B} = \{(u,v,w) \ | \ w = uv, -1 \leq u,v \leq 1\}$. 
The problem~\eqref{eq:SVconesineq} is equivalent to 
\begin{align*} \label{eq:origpbl}
\lambda^* = \min_{u \in P, v \in Q, w \in \mathbb{R}^{d}} & \sum_{k=1}^d A(i_k,j_k) w_{k}  
 \text{ such that } 
(u_{i_k} , v_{j_k}, w_{k}) \in \mathcal{B} \text{ for } (i_k,j_k) \in \Omega, 
\|u\| \leq 1 \text{ and } \|v\| \leq 1 . 
\end{align*} 
The constraints $w_{k} = u_{i_k} v_{j_k}$ makes the problem non-convex. However, it can be relaxed using the McCormick envelope, which uses the smallest convex set containing $\mathcal{B}$: 
Constraints of the form $w = u v$ are relaxed to  
\begin{equation*}
 w  \leq -u +  v +1 ,  \; 
 w  \leq   u  -  v +1 ,  \;
 w  \geq -  u - v  -1 , \; 
 w  \geq  u       +  v -   1 . 
\end{equation*} 
The curve $(u,v,uv)$ is sandwiched between four hyperplanes, each of them containing $3$ vertices of the hypercube $[-1,1]^3$. 

The McCormick relaxation optimizes over this (linear) convex envelope, and hence obtains a lower bound for $\lambda^*$. From this solution, Gurobi will subdivide the feasible set into smaller pieces, using a branch and bound approach. 
For each of the $d$ non-zero entries of $A$, we might have to branch up to a desired precision, and the worst-case complexity of this procedure is in 
$O\left( \left(\frac{1}{\epsilon}\right)^{d} \text{poly}(m,n) \right)$ where $\epsilon$ is the desired accuracy.  In practice, Gurobi can avoid exploring a large part of the domain, because of the branch-and-bound approach. Note that when $A$ has a few non-zero entries, there are fewer non-convex terms and hence Gurobi is more likely to solve such problems faster. 

Recall that in case of polyhedral cones, we can rewrite problem $\SV(A,P,Q)$
in terms of the generators of the cones, as in
 \begin{equation}\label{eq:SVcones_generators2}
    \min_{\scriptsize \begin{array}{l}x\ge 0,\,\Vert Gx\Vert \leq 1,\\y\ge 0,\,\Vert Hy\Vert \leq 1\end{array}}  \langle Gx,AHy\rangle.
\end{equation}
This is again a non-convex quadratic problem that can be solved with Gurobi, but from the experiments we consistently observe that solving \eqref{eq:SVconesineq} in $(u,v)$ is faster than solving \eqref{eq:SVcones_generators2} in $(x,y)$. This behavior depends on the fact that $G^\top A H$ is usually less sparse than $A$, so the McCormick relaxation introduces more variables and the computational cost rises considerably.

\subsection{Heuristic algorithms}   \label{sec:heuristics}

In this section, we propose two heuristic algorithms (that is, algorithms that come with no global optimality guarantees) to tackle $\SV(A,P,Q)$: one based on alternating optimization (Section~\ref{sec:AO}), and one based on fractional programming (Section~\ref{sec:fracprog}).

{We point out that a method based on alternating optimization has already been explored in \cite[Section\,6]{seeger2023PSV} for the problem $\PSV(A)$ (with $A$ having at least one negative entry) by exploiting the equivalence with the best nonnegative rank-one approximation problem~\eqref{th:rank1NF}. }

\subsubsection{Alternating projection with extrapolation} \label{sec:AO}

Given $A$, $P$ and $Q$, we want to solve 
\[
\min_{u \in P, v \in Q} \; \langle u,Av\rangle
\quad \text{ such that } \quad \|u\| = \|v\| = 1. 
\]
A standard, simple and often effective optimization strategy is block coordinate descent. Recall that one can easily check whether the optimal objective function value is nonnegative; see Proposition~\ref{pro:nonnegative_case_SVcones}. If it is negative, then one can relax the constraints $\|u\| = \|v \| = 1$ to $\|u\|$ $\leq 1$, $\|v \| \leq 1$ to make the feasible set convex. 
In this case, we are facing a bi-convex problem, that is, the problem is convex in $u$ when $v$ is fixed, and vice versa. Hence, it makes sense to use two blocks of variables: $u$ and $v$, also known as alternating optimization (AO). Given an initial $v$, it simply alternates between the optimal update of $u$ and $v$, which are convex optimization problems. 

In our case, the subproblems are convex and the feasible sets are compact, hence AO is guaranteed to have a subsequence converging to a critical pair~\cite{grippo2000convergence}.

\paragraph{Subproblems in $u$ and $v$:}

The subproblem in $u$ is expressed as
\[u  = \argmin_{x \in P} \langle x,Av\rangle
 \; \text{ such that } \; \|x\| \,  \leq  \, 1,\]
and the subproblem in $v$ is analogous. 
This is a classical optimization problem as discussed in Section~\ref{sec:simple_cases_generators}. 
In the case of polyhedral cones,  we use Gurobi to solve the two subproblems. 

In some special cases, the subproblems in $u$ and $v$ can be solved in closed form. 
In fact, in the case  $P = \mathbb{R}^{m}_+$, given the vector $c = Av$, we need to solve 
$\min_{x \geq 0, \|x\| = 1} \langle x,c\rangle$. It can be easily checked that the optimal solution is $x = \frac{\min(c,0)}{\|\min(c,0)\|}$ 
if $c \ngeq 0$, otherwise $x$ has a single non-zero entry equal to one at a position $i \in \argmin_j c_j$. 

When $P$ and $Q$ are not polyhedral, the same algorithm can still be applied, as long as we can explicitly express the projection maps on both cones, or if we have a convergent method to solve the convex subproblems in $u,v$.

\paragraph{Extrapolation} AO can sometimes be relatively slow to converge. To accelerate convergence, we use extrapolation. After each update of $u$ (and similarly for $v$), we define the extrapolated point $u_e = u + \beta (u - u_p)$ where $u_p$ is the previous iterate, and $\beta \in [0,1]$ is a parameter. Note that the extrapolated point does not need to be feasible. 
To guarantee convergence, we restart the scheme when the objective increases, and also decrease $\beta$ by a factor $\eta$ (we will use $\eta = 2$). When the objective decreases, we slightly increase $\beta$ by a factor $\gamma$ to reinforce the extrapolation effect (we will use $\gamma = 1.05$). 
This is a similar strategy as used in~\cite{ang2019accelerating}. 

Algorithm~\ref{alg:eao} summarizes our proposed extrapolated AO (E-AO). \\
\begin{algorithm}[h!]
\caption{Extrapolated Alternating Optimization (E-AO) to solve $\SV(A,P,Q)$} \label{alg:eao}
\begin{algorithmic}[1]

\REQUIRE Matrix $A \in \R^{m\times n}$,  {poyhedral cones} $P \subseteq \mathbb{R}^{m}$ and $Q \subseteq \mathbb{R}^{n}$, initial point 
$v_0 \in Q$ with $\|v_0\| = 1$, maximum number of iteration $K$, stopping criterion $\delta \ll 1$, extrapolation parameters $\beta \in (0,1]$,  $\eta > \gamma > 1$. (Default: $K=500$, $\delta = 10^{-6}$, $\beta = 0.5, \eta = 2, \gamma = 1.05$.)

\ENSURE An approximate solution to $\min_{u \in P, v \in Q} \langle u,Av\rangle \text{ such that } \|u\|,\|v\|$ $ \leq $ $1$. 

\medskip 

\State $u = 0$, $v = 0$, $v_e = v_0$, $k = 1$, $\beta_p = \beta$, $u_p = u$, $v_p = v$, rs $=0$. 

\WHILE{
$k \leq K$ 
and 
$\Big( \text{rs} = 1 \text{ or } \|u-u_p\| \geq \delta 
\text{ or }  \|v-v_p\| \geq \delta $ 
or 
$\left(k \leq 3 \text{ or } e_{k-2}-e_{k-1} \geq \delta e_{k-2}\right) \Big)$
}

  \State $u_p = u$. \emph{\% Keep previous iterate in memory} 
  
 \State  $u \; = \; \argmin_{x \in P} \langle x,Av_e\rangle 
 \; \text{ such that } \; \|x\|$ $\leq $ $1.$ 

\State $u_e = u + \beta (u-u_p)$. \emph{\% Extrapolated point}

  \State $v_p = v$. \emph{\% Keep previous iterate in memory}  
  
 \State  $v \; = \;\argmin_{y \in Q} \langle u_e,Ay\rangle 
 \; \text{ such that } \; \|y\|$ $ \leq $ $1.$

\State $v_e = v + \beta (v-v_p).$ \emph{\% Extrapolated point}

 \State  \textbf{\emph{\% Restart scheme when the objective increases}} 
 
\State $e_k = u^\top A v$, rs $=0$.

\IF {$k \geq 2$ and $e_k > e_{k-1}$ and $\beta > 0$}

\State $u = u_p$, 
$v = v_p$, 
$v_e = v_p$, 
$\beta_p = \frac{\beta}{\eta}$, 
$\beta = 0$, 
rs $=1$,
$e_k = e_{k-1}$. \emph{\% Next step will not extrapolate}

\ELSE

\State $\beta = \min( 1, \gamma  \beta_p)$, $\beta_p = \beta$. 

\ENDIF

\State $k \leftarrow k+1$. 

\ENDWHILE  

\end{algorithmic}
\end{algorithm}

In our experiments, we will use multiple random initializations: 
we generate $u_0$ at random using the Gaussian distribution (\texttt{u$\_$0 = randn(m,1)}); note that $u_0$ does not necessary belong to $P$. Then we get $v_0 \in Q$ by solving\footnote{Recall that, if the optimal $v^* = 0$, it means the optimal solution for $\|v\|=1$ is an extreme ray of $Q$ minimizing the objective.} $\min_{v \in Q, \|v\| \leq 1} \langle u,Av\rangle $. 
Also, because of the automatic tuning of $\beta$, E-AO is not too sensitive to its initial value; we will use $\beta = 0.5$.

\subsubsection{A sequential regularized partial linearization algorithm} \label{sec:fracprog}

Recently, de Oliveira, Sessa, and Sossa~\cite{de2023computing}  proposed a sequential regularized partial linearization (SRPL) algorithm for computing the maximal angle between two linear images of symmetric cones (LISCs). 
It is straightforward to adapt that method for solving $\SV(A,P,Q)$. For simplicity of the presentation, we only describe the algorithm for polyhedral cones 
which are particular instances of LISCs. 
Let $A\in\mathbb R^{m\times n}$, and $P\subset\mathbb R^m$ and $Q\subset\mathbb R^n$ be polyhedral cones generated by $G$ and $H$, respectively, as in~\eqref{polyhedral}. 
In this section, we also assume that $P$ and $Q$ are pointed. That $P$ is pointed means that $P\cap -P=\{0\}$ which is equivalent to the property that
$x\geq 0$ and $Gx=0$ imply $x=0$.

The problem $\SV(A,P,Q)$ can be equivalently reformulated as 
\begin{equation}\label{frac1}
\min\limits_{u\in P,\,v\in Q}\;\displaystyle\frac{\langle u,Av\rangle}{\Vert u\Vert \Vert v\Vert}\quad \mbox{such that}\quad u\neq 0\mbox{ and }v\neq 0.
\end{equation}
The $\ell_2$ norm constraint in $\SV(A,P,Q)$ is removed 
by using a fractional objective function in~\eqref{frac1}. 
By making the change of variables $u=Gx$ and $v=Hy$, the objective function does not depend on the normalizations on $x,y$, so we can impose instead $\langle {\bf 1},x\rangle =\langle {\bf 1},y\rangle  =1$. 
Then, by denoting $\Delta_d:=\{x\in\mathbb R^d: x\geq 0,\, \langle {\bf 1},x\rangle = 1\}$, the probability simplex in $\mathbb R^d$, \eqref{frac1} becomes 
\begin{equation}\label{frac2}
\min\limits_{x\in\Delta_p,\,y\in\Delta_q}\;\displaystyle\Phi(x,y):=\frac{\langle G x, A Hy\rangle}{\Vert Gx\Vert \Vert Hy\Vert}.
\end{equation}
{The denominator of \eqref{frac2} does not vanish for any $(x,y)\in\Delta_p\times\Delta_q$ because of the pointedness assumption on the cones.} The SRLP algorithm, described in Algorithm\,\ref{alg:frac}, relies on solving the fractional program \eqref{frac2} by following the approach given by Dinkelbach in \cite{dinkelbach}. That is, the problem \eqref{frac2} is reformulated as a parametric optimization problem on $\Delta_p\times\Delta_q$ with objective function $f_\delta(x,y)=\langle G x, A Hy\rangle-\delta \Vert Gx\Vert \Vert Hy\Vert$ with $\delta\in\mathbb R.$ The method to solve this parametrized problem consists on linearizing  $f_\delta(x,y)$ with respect to each variable and solving regularized linear programs for each variable; see Algorithm~\ref{alg:frac}. 
\begin{algorithm}[h!]
\caption{Sequential Regularized Partial Linearization (SRPL)}\label{alg:frac}
\begin{algorithmic}[1]

\REQUIRE Matrix $A\in\mathbb R^{m\times n}$, polyhedral cones $P=G(\mathbb R^p_+)$ and $Q=H(\mathbb R^q_+)$; initial points $x^0\in\Delta_p$, $y^0\in\Delta_q$; prox-parameters  $\mu_1,\mu_2\geq0$; line-search algorithm parameters $\beta>0$, $0<\alpha<1$, $0<\rho<1$; maximum number of iteration $K$, stopping criterion $\delta \ll 1$. (Default:  $\beta =1$, $\alpha = .001$, $\delta=10^{-6}$, $K=5000$, $\rho = .2$.)

\ENSURE An approximate solution to $\min_{u \in P, v \in Q} \langle u, A v\rangle \text{ such that } \|u\|=\|v\| = 1$.

\State Set $k:=0$.
\State Set $\displaystyle \delta_k:=\frac{\langle Gx^k,A Hy^k\rangle}{\Vert G x^k\Vert \Vert H y^k\Vert}.$
\vspace{0.2cm}
\State Let $L^k_1(x) := \left\langle Gx,AHy^k-\delta_k\Vert Gx^k\Vert^{-1}\Vert Hy^k\Vert Gx^k\right\rangle$. Compute a solution $\tilde x^k$ to the convex  program
\begin{equation}\label{lin1}
\min_{x \in \Delta_p}  L^k_1(x) + \frac{\mu_1}{2} \|x - x^k\|^2. 
\end{equation}

\State Let $L^k_2(y) :=  \left\langle Hy,A^\top G x^k-\delta_k\Vert G x^k\Vert\Vert Hy^k\Vert^{-1} Hy^k\right\rangle$. Compute a solution $\tilde y^k$ to the convex program
 \begin{equation}\label{lin2}
\min_{y\in \Delta_q}  L^k_2(y) + \frac{\mu_2}{2} \|y - y^k\|^2 . 
\end{equation}

\State Let $d_1^k:=\tilde x^k-x^k$ and $d_2^k:=\tilde y^k-y^k$.
\vspace{0.2cm}

\State If ($\vert L^k_1( d_1^k)\vert<\delta$ and $\vert L^k_2( d_2^k)\vert<\delta$) or $k\geq K$ terminate. \newline Otherwise, let $t_k:= \beta \rho^{\ell_k}$, where $\ell_k$ is the smallest nonnegative integer $\ell$ such that
$$\Phi(x^k+t^k d^k_1,y^k+t^k d^k_2) \leq \Phi(x^k,y^k)  + 
\alpha t_k \frac{L^k_1( d_1^k)+L^k_2( d_2^k)}{\Vert Gx^k\Vert\Vert Hy^k\Vert}
$$
Set $\left(x^{k+1},y^{k+1}\right):=(x^k,y^k)+t_k(d^k_1,d^k_2)$. Go to step 2.
\end{algorithmic}
\end{algorithm}
{Note that the solution of problem \eqref{lin1} in Algorithm~\ref{alg:frac} can be obtained by solving the following projection problem onto the probability simplex: 
\begin{equation*}
\min\limits_{x\in \Delta_p}  \frac{1}{2} \Big\|x - \big(x^k-\frac{c_k}{\mu_1}\big) \Big\|^2
\quad \mbox{with}\quad 
c_k :=  G^\top(AHy^k-\delta_k\Vert Gx^k\Vert^{-1}\Vert Hy^k\Vert Gx^k).
\end{equation*}
An analogous observation holds for problem \eqref{lin2}.}

\section{Numerical experiments} \label{sec:numexp}

To solve problem $\SV(A,P,Q)$, we have described four methods: the brute-force active set (BFAS, 
Algorithm~\ref{alg:active_set}), a non-convex quadratic solver for polyhedral cones with Gurobi (Gur), the extrapolated alternating optimization method 
(E-AO, Algorithm~\ref{alg:eao}) and the sequential regularized partial linearization (SRPL, Algorithm~\ref{alg:frac}). 

Among them, only BFAS and Gur can verifiably solve the problem exactly (up to a fixed tolerance). E-AO and SRPL are instead designed to be fast heuristics, but they only converge to stationary points. BFAS can also be considered a heuristic method when bound by a fixed time limit. Gurobi, instead, has a built-in fast auxiliary heuristic method which is used to speed up the branch and bound method, so Gur is also able to provide good upper bounds to the optimal solution in a relatively short time.

 All experiments are implemented in MATLAB (R2024a) and run on a laptop with an 13th Gen Intel Core™ i7-1355U and 16 GB RAM. For the experiments involving the software Gurobi, we use the version 11.0.0. 
The codes, data and results for all algorithms and experiments can be found in the repository \url{https://github.com/giovannibarbarino/coneSV/}.

\subsection{Schur cone and nonnegative orthant} 

The Schur cone $\mc H$ in $n$ dimensions is defined by 
$$
\mc H:= \left\{x\in \f R^n \,{\Bigg |}\, 
\sum_{i=1}^k x_i \ge 0 \quad\forall 1\le k<n,\,\,
 \langle {\bf 1}, x\rangle  = 0\right\}\subset\R^n.
$$
First, we derive the maximal angle and the antipodal pairs of $\MA(\mc H,\f R_+^n)$, so that we have a ground truth over which to compare the algorithms. 

\begin{lemma}\label{lem:max_angle_schur_posorth}
    If $\mc H$ is the Schur cone in $n$ dimensions, then \[
\min_{
	\begin{array}{c}
	x\in \mc H, y\in \f R^n_+\\
	\|x\|=\|y\|=1
	\end{array}
}
\langle x,y\rangle = \langle x^*,y^*\rangle = - \sqrt {1-\frac {1}n}, 
\qquad 
y^* = e_n, \quad x^* = \sqrt{\frac n{n-1}} \left(\frac {\bf 1}n -e_n\right),
\]
where $(x^*,y^*)$ is the only antipodal pair.
\end{lemma}
\begin{proof}
Let $\lambda^*$ denote the optimal value of $\MA(\mc H,\R^n_+)$. Notice that $\mc H$ is not contained in $\R^n_+$ but it is contained in the subspace ${\bf 1}^\perp$, so
\[
 0>\lambda^* \ge 
 - \max_{\scriptsize
 	\begin{array}{c}
	x\in {\bf 1}^\perp, y\in \f R^n_+\setminus\R_+{\bf 1}\\
	\|x\|=\|y\|=1
	\end{array}
 }
 \langle x,y\rangle
 = -1-\frac{1}{2}\min_{\scriptsize
 	\begin{array}{c}
	y\in \f R^n_+\setminus\R_+{\bf 1}\\
	\|y\|=1
	\end{array}
 } 
  \min_{\scriptsize
 	\begin{array}{c}
	x\in {\bf 1}^\perp\\
	\|x\|=1
	\end{array}
 }
 \Vert y-x\Vert^2.
 \]
Observe that for $y\in\R^n_+\setminus \R_+{\bf 1}$ such that $\Vert y\Vert=1$, the last minimization problem has the orthogonal projection of $y$ onto ${\bf 1}^\perp\cap S_n$ as its unique solution. This is given by ${\rm Proj}_{{\bf 1}^\perp}(y)/\Vert {\rm Proj}_{{\bf 1}^\perp}(y)\Vert$,
    where ${\rm Proj}_{{\bf 1}^\perp}(y)=y-(\langle y,{\bf 1}\rangle/n) {\bf 1}$ is the projection of $y$ onto ${\bf 1}^\perp$, see \cite{bauschke-projectingonto2018}. Then,
    $$\lambda^*\geq
 -\max_{\scriptsize
 	\begin{array}{c}
	y\in \f R^n_+,\,\|y\|=1
	\end{array}
 } 
 \left(1-\frac{1}{n}\langle y,{\bf 1}\rangle^2\right)^{1/2}
 =
 -
 \left(1-\frac{1}{n}\left[\min_{\scriptsize
 	\begin{array}{c}
	y\in \f R^n_+,\,\|y\|=1
	\end{array}
 } \langle y,{\bf 1}\rangle\right]^2\right)^{1/2}= - \sqrt {1-\frac {1}n}.
 $$
The proof concludes by noticing that $(x^*,y^*)$ achieves the above bound, and that the solution set of the last minimization problem are the canonical vectors of $\R^n$, where $y^*=e_n$ is the only vector that produces a feasible $x^*$ in $\mc H$.  
\end{proof}

 Table \ref{table:Gur_BFAS_Schur_orthant} reports the optimal value found by the algorithms Gur and BFAS regarding the problem of finding the largest angle between the Schur cone and the positive orthant in dimension $n$ for $n=5,10,20,50,100,200,500$. When the algorithms terminate in less than $60$ seconds, we report the elapsed time, otherwise we report the best value found in $60$ seconds. The numbers in bold are the optimal angle found, whenever they coincide with the exact angle up to a tolerance of $10^{-5}\pi$, and the elapsed time is under $60$ seconds. Sometimes, the algorithms may find larger angles than the exact ones, but it has to be attributed to rounding errors.

\begin{table}[ht!]
\caption{Numerical comparison for Gur and BFAS for different dimensions for the problem of finding the maximal angle between the Schur cone and the positive orthant cone. The table reports the optimal objective function values (in terms of angles, which is more interpretable) found in the time limit (60 seconds) and the actual elapsed time. We also report the exact value for each problem. 
\label{table:Gur_BFAS_Schur_orthant}}
\begin{center}
\begin{tabular}{c|ccccccc}
$n$ \hspace{-0.2cm} & \hspace{-0.2cm} $5 $\hspace{-0.2cm} & \hspace{-0.2cm} $10 $\hspace{-0.2cm} & \hspace{-0.2cm} $20 $\hspace{-0.2cm} & \hspace{-0.2cm} $50 $\hspace{-0.2cm} & \hspace{-0.2cm} $100 $\hspace{-0.2cm} & \hspace{-0.2cm} $200 $\hspace{-0.2cm} & \hspace{-0.2cm} $500 $\\ \hline  \hline 
exact  & \hspace{-0.2cm} $0.852416 \pi$\hspace{-0.2cm} & \hspace{-0.2cm} $0.897584 \pi$\hspace{-0.2cm} & \hspace{-0.2cm} $0.928217 \pi$\hspace{-0.2cm} & \hspace{-0.2cm} $0.954833 \pi$ \hspace{-0.2cm} & \hspace{-0.2cm}$0.968116 \pi$\hspace{-0.2cm} & \hspace{-0.2cm} $0.977473 \pi$\hspace{-0.2cm} & \hspace{-0.2cm} $0.985760 \pi$\\ \hline  \hline 
Gur & \hspace{-0.2cm}  $\bm{0.852416} \pi$ \hspace{-0.2cm} & \hspace{-0.2cm} $ \bm{0.897584} \pi$\hspace{-0.2cm} & \hspace{-0.2cm} $ \bm{0.928218} \pi$\hspace{-0.2cm} & \hspace{-0.2cm} $ \bm{0.954833} \pi$\hspace{-0.2cm} & \hspace{-0.2cm} $ \bm{0.968116} \pi$\hspace{-0.2cm} & \hspace{-0.2cm} $ \bm{0.977473} \pi$\hspace{-0.2cm} & \hspace{-0.2cm} $ \bm{0.985756} \pi$\\ 
\hspace{-0.2cm} & \hspace{-0.2cm} $  \bm{0.1134} $ s\hspace{-0.2cm} & \hspace{-0.2cm} $  \bm{0.2016} $ s\hspace{-0.2cm} & \hspace{-0.2cm} $ \bm{20.1493} $ s\hspace{-0.2cm} & \hspace{-0.2cm} $60 ^*$ s\hspace{-0.2cm} & \hspace{-0.2cm} $60 ^*$ s\hspace{-0.2cm} & \hspace{-0.2cm} $60 ^*$ s\hspace{-0.2cm} & \hspace{-0.2cm} $60 ^*$ s\\ \hline
BFAS & \hspace{-0.2cm} $ \bm{0.852416} \pi$\hspace{-0.2cm} & \hspace{-0.2cm} $ \bm{0.897584} \pi$\hspace{-0.2cm} & \hspace{-0.2cm} $ 0.750000 \pi$\hspace{-0.2cm} & \hspace{-0.2cm} $ 0.750000 \pi$\hspace{-0.2cm} & \hspace{-0.2cm} $ 0.750000 \pi$\hspace{-0.2cm} & \hspace{-0.2cm} $ 0.750000 \pi$\hspace{-0.2cm} & \hspace{-0.2cm} $ 0.750000 \pi$\\ 
\hspace{-0.2cm} & \hspace{-0.2cm} $  0.3310 $ s\hspace{-0.2cm} & \hspace{-0.2cm} $ 48.3153 $ s\hspace{-0.2cm} & \hspace{-0.2cm} $60 ^*$ s\hspace{-0.2cm} & \hspace{-0.2cm} $60 ^*$ s\hspace{-0.2cm} & \hspace{-0.2cm} $60 ^*$ s\hspace{-0.2cm} & \hspace{-0.2cm} $60 ^*$ s\hspace{-0.2cm} & \hspace{-0.2cm} $60 ^*$ s 
\end{tabular} 
 \end{center} \vspace{-0.3cm} 
 \end{table}
 
We observe that Gur outperforms BFAS both in speed and accuracy. In fact, Gur guarantees to have found the optimal solution up to dimension $20$ in less than $60$ seconds, and even when its computational time exceeds a minute, it finds the exact solution up to dimension at least $500$. BFAS instead can only solve the problem exactly up to dimension $10$ and it is way slower than Gur.  

\begin{figure}[ht!]
\begin{center} 
\begin{tabular}{cc}
Exp.~1: Schur cone - Nonnegative Orthant & Exp.~2:  Schur cone - Schur cone 
\\
\hspace{-.3 cm} \includegraphics[width=.48\textwidth]{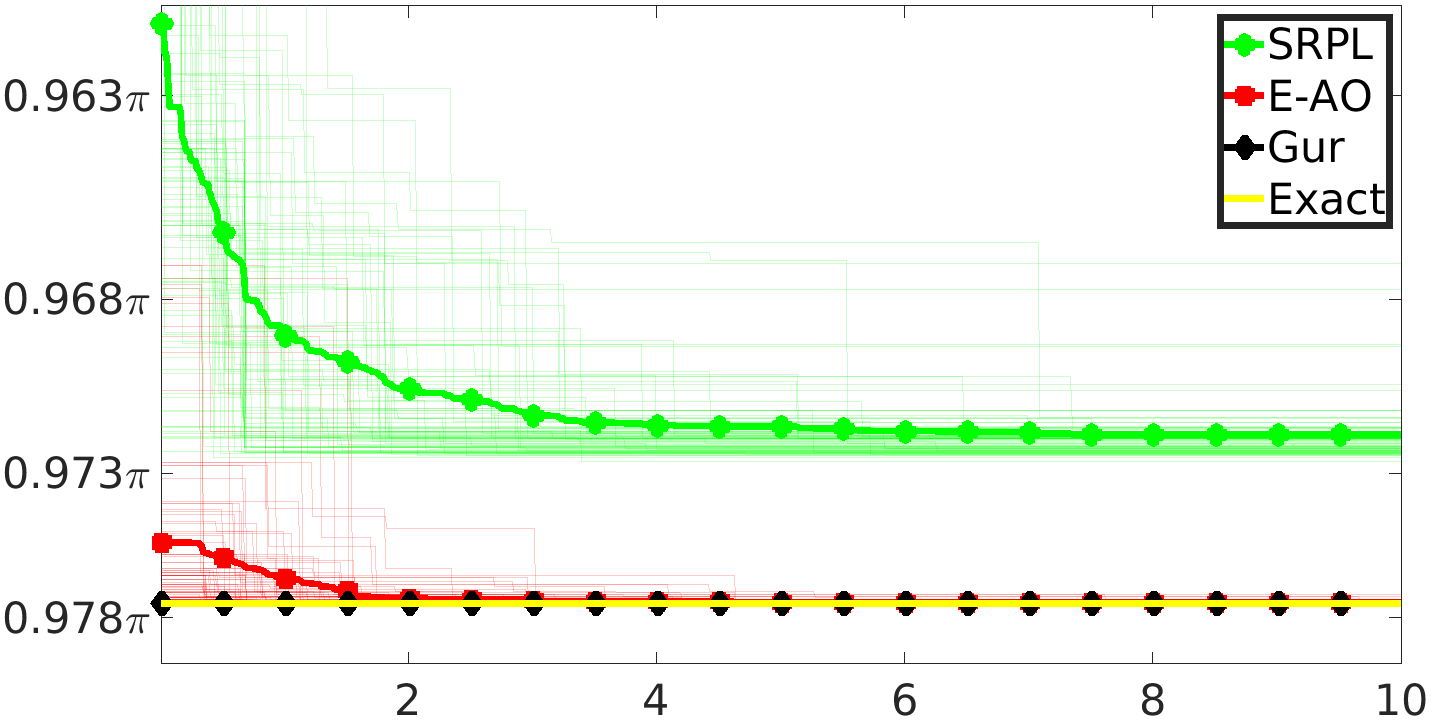}&
  \includegraphics[width=.48\textwidth]{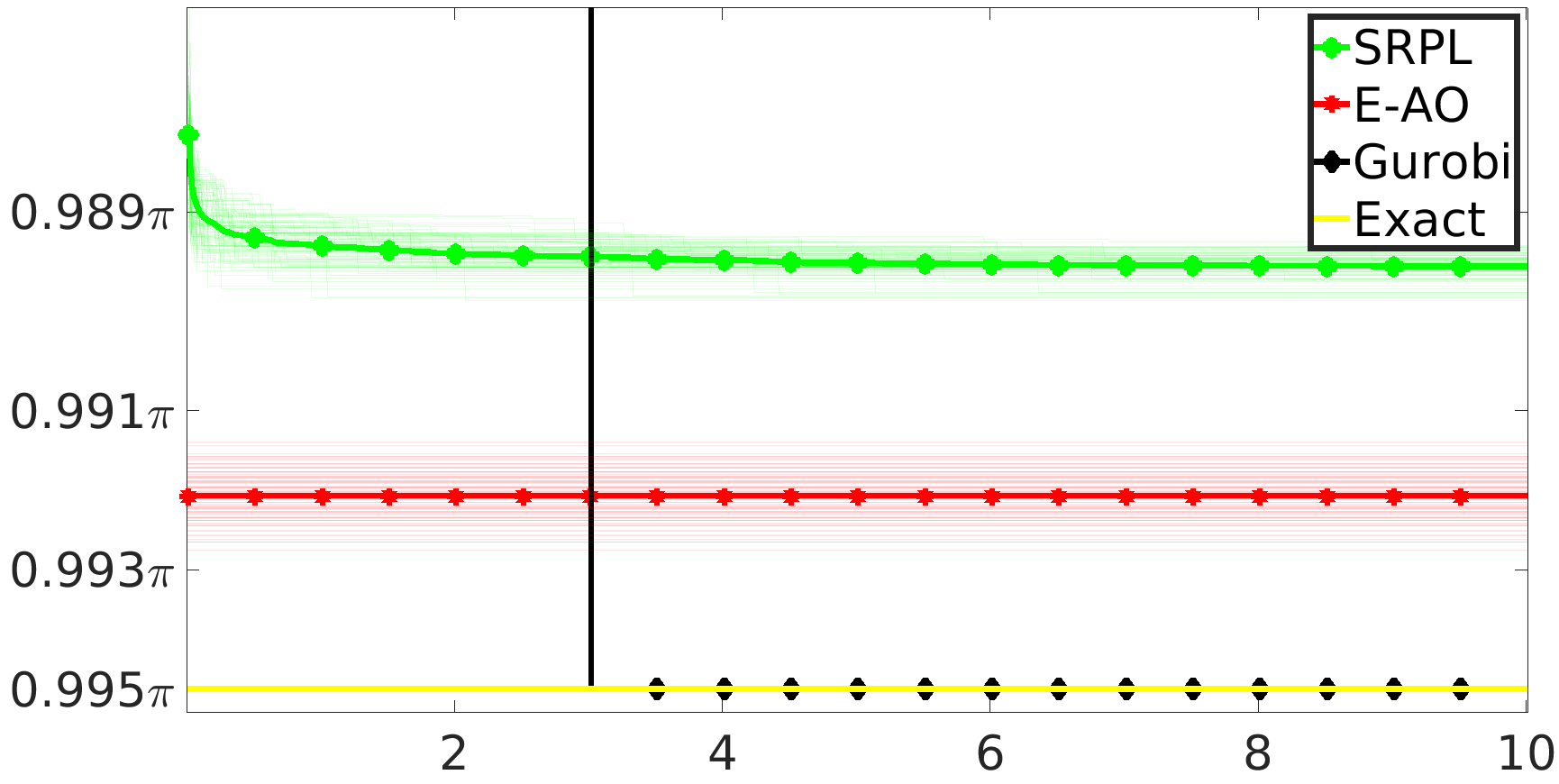}\end{tabular}
\caption{E-AO, SRPL, BFAS and Gurobi compared on the problem of finding the maximal angle between the Schur cone and the nonnegative orthant (left image), and between the Schur cone and itself (right image) in dimension $n=200$. For E-AO and SRPL, $100$ iterations from randomly generated points are plotted in lighter colors, and their average with a thick line. The x-axis represents time in seconds.
\label{fig:Schur-PosOrthant}}  
\end{center}
\end{figure}   
In the left image of Figure \ref{fig:Schur-PosOrthant}, we report a comparison of the four methods over a timespan of $10$ seconds for the same problem in dimension $n=200$. Since a single execution of E-AO and SRPL usually takes between $0.3$ and $2$ seconds, we restart the algorithms with a new random initial point and keep only the best solution, until we reach the mark of $10$ seconds. We perform $100$ of these $10$ seconds run for both heuristic algorithms and plot the average value over time, in addition to all performed runs in the background. The parameters used in SRPL are $ \mu_1= 0.25$, $\mu_2=0.01$.  

We observe that Gur immediately reaches the optimum value, also met by E-AO in less than $2$ seconds on average. SRPL, instead, does not converge to the optimal value in any run, and it seems unable to break the $0.972\pi$ barrier. 
Finally, BFAS is not represented in the plot since its best value at $10$ seconds is equal to $0.75\pi$.

\subsection{Schur cone with itself} 

{The Schur cone has a more involved structure than the nonnegative orthant. So, in this section, we test the performance of our algorithms when both cones on the problem $\MA(P,Q)$ are the Schur cones, that is, $P=Q=\mc H$. The precise optimal value of $\MA(\mc H,\mc H)$ was computed in \cite[Proposition\,2]{gourion2010critical}  and is equal to $\cos\left(\frac{n-1}{n}\pi\right)$.

Table~\ref{table:Gur_BFAS_Schur_Schur} reports the optimal value found by the algorithms Gur and BFAS regarding the problem of finding the largest angle between the Schur cone and itself in dimension $n \in \{5, 10, 20, 50, 100, 200, 500\}$. We use exactly the same setting as in the previous section (Table~\ref{table:Gur_BFAS_Schur_orthant}). 

We observe that Gur outperforms again BFAS both in speed and accuracy. In fact, even when its computational time exceeds a minute, it finds the exact solution up to dimension $50$, and also for dimensions $200, 500$. The error in dimension $100$, though, tells us that it cannot always be blindly trusted. BFAS instead can only solve the problem exactly up to dimension $5$, and the approximations at $60$ seconds is only reliable up to dimension $10$.  Sometimes, the algorithms may find larger angles than the exact ones, but it has to be attributed to rounding errors.

Notice that in this case, Gurobi results are less accurate than those of the previous problem. This is possibly due to the fact that the matrix $H^\top H$, where $H$ are the generators of $\mc H$, is less sparse than the matrix $H^\top I$, where $I$ represents the generators of the nonnegative orthant.
 	
In the right image of Figure~\ref{fig:Schur-PosOrthant}, we report a comparison of the four methods over a timespan of $10$ seconds for the same problem in dimension $n=200$. Again, E-AO and SRPL are usually restarted until we reach the $10$ seconds mark, and we perform $100$ of these $10$ seconds runs for both heuristic algorithms, plotting the average value over time, and all performed runs in the background. The parameters used by  SRPL in this case are $\mu_1=\mu_2=1$.

This time, Gur is the only method to correctly converge to the optimum value in the $10$ seconds (actually, after only $3$ seconds). E-AO almost immediately converges to $0.992\pi$ but in no run it manages to reach the optimal value. SRPL offers a more diverse plot, but it always converges to an even larger value. Finally, BFAS is again not represented in the plot since its best value at $10$ seconds is equal to $0.75\pi$. 

\begin{table}[H]
\caption{Numerical comparison for Gurobi and BFAS for different dimensions for the problem of finding the maximal angle between the Schur cone and itself. The table reports the optimal objective function values found in the time limit (60 seconds) and the actual elapsed time. We also report the exact value for each problem. \label{table:Gur_BFAS_Schur_Schur}}
\begin{center}
\begin{tabular}{c|ccccccc}
$n$ \hspace{-0.2cm} & \hspace{-0.2cm}  $5 $\hspace{-0.2cm} & \hspace{-0.2cm}  $10 $\hspace{-0.2cm} & \hspace{-0.2cm}  $20 $\hspace{-0.2cm} & \hspace{-0.2cm}  $50 $\hspace{-0.2cm} & \hspace{-0.2cm}  $100 $\hspace{-0.2cm} & \hspace{-0.2cm}  $200 $\hspace{-0.2cm} & \hspace{-0.2cm}  $500 $\\ \hline  \hline 
exact & \hspace{-0.2cm}  $0.800000 \pi$\hspace{-0.2cm} & \hspace{-0.2cm}  $0.900000 \pi$\hspace{-0.2cm} & \hspace{-0.2cm}  $0.950000 \pi$\hspace{-0.2cm} & \hspace{-0.2cm}  $0.980000 \pi$\hspace{-0.2cm} & \hspace{-0.2cm}  $0.990000 \pi$\hspace{-0.2cm} & \hspace{-0.2cm}  $0.995000 \pi$\hspace{-0.2cm} & \hspace{-0.2cm}  $0.998000 \pi$\\ \hline  \hline 
Gur & \hspace{-0.2cm}  $ \bm{0.800001} \pi$\hspace{-0.2cm} & \hspace{-0.2cm}  $ \bm{0.900000} \pi$\hspace{-0.2cm} & \hspace{-0.2cm}  $ \bm{0.950000} \pi$\hspace{-0.2cm} & \hspace{-0.2cm}  $ \bm{0.980000} \pi$\hspace{-0.2cm} & \hspace{-0.2cm}  $ 0.936315 \pi$\hspace{-0.2cm} & \hspace{-0.2cm}  $ \bm{0.994996} \pi$\hspace{-0.2cm} & \hspace{-0.2cm}  $ \bm{0.998011}  \pi$\\ 
\hspace{-0.2cm} & \hspace{-0.2cm}  $  \bm{0.2508} $ s\hspace{-0.2cm} & \hspace{-0.2cm}  $60 ^*$ s\hspace{-0.2cm} & \hspace{-0.2cm}  $60 ^*$ s\hspace{-0.2cm} & \hspace{-0.2cm}  $60 ^*$ s\hspace{-0.2cm} & \hspace{-0.2cm}  $60 ^*$ s\hspace{-0.2cm} & \hspace{-0.2cm}  $60 ^*$ s\hspace{-0.2cm} & \hspace{-0.2cm}  $60 ^*$ s\\ \hline
BFAS & \hspace{-0.2cm}  $ \bm{0.800000} \pi$\hspace{-0.2cm} & \hspace{-0.2cm}  $ \bm{0.900000} \pi$\hspace{-0.2cm} & \hspace{-0.2cm}  $ 0.859157 \pi$\hspace{-0.2cm} & \hspace{-0.2cm}  $ 0.804087 \pi$\hspace{-0.2cm} & \hspace{-0.2cm}  $ 0.750000 \pi$\hspace{-0.2cm} & \hspace{-0.2cm}  $ 0.750000 \pi$\hspace{-0.2cm} & \hspace{-0.2cm}  $ 0.750000 \pi$\\ 
\hspace{-0.2cm} & \hspace{-0.2cm}  $  0.3856 $ s\hspace{-0.2cm} & \hspace{-0.2cm}  $60 ^*$ s\hspace{-0.2cm} & \hspace{-0.2cm}  $60 ^*$ s\hspace{-0.2cm} & \hspace{-0.2cm}  $60 ^*$ s\hspace{-0.2cm} & \hspace{-0.2cm}  $60 ^*$ s\hspace{-0.2cm} & \hspace{-0.2cm}  $60 ^*$ s\hspace{-0.2cm} & \hspace{-0.2cm}  $60 ^*$ s
\end{tabular} 
 \end{center} 
 \end{table}

\subsection{Computing the biclique number}

Given a biadjacency matrix $B\in \{0,1\}^{m\times n}$, solving the NP-hard  maximum edge biclique problem is equivalent to solving the Pareto singular value problem $\PSV(-M)$ where $M = B-(1-B)d$ and  $d\ge \max(m,n)$; 
see Theorems~\ref{th:rank1NF} and~\ref{th:SVcones_NPhard}.  Here we thus test all four algorithms on four bipartite graphs taken from the dataset in \cite{MaximumBicliqueBenchmark}. 
They correspond to the files \url{https://github.com/giovannibarbarino/coneSV/Biclique_matrix_n.txt} in the repository for $n=1,2,3,4$. 
All graphs have been randomly generated with a fixed edge density, and then a biclique has been added (planted) to them. 
In particular,
\begin{itemize}
    \item the first graph is a $100 \times 100$ graph with density $0.2$ and planted biclique of size $50 \times 50 = 2500$,
    
    \item the second graph is a $300 \times 300$ graph with density $0.3$ and planted biclique of size $2 \times 55=110$,
    
    \item the third graph is a $100 \times 100$ graph with density $0.71$ and planted biclique of size $80 \times 80 = 6400$,
    
    \item the fourth graph is a $10000 \times 100$ graph with density $0.03$ and planted biclique of size $22 \times 2 = 44$.
\end{itemize}
All algorithms are stopped after $10$ seconds, and the results (rounded to the nearest integer value) are reported in Table \ref{table:Gur_BFAS_biclique}. E-AO and SRPL are restarted until we reach the $10$ seconds mark, and we perform $100$ of these $10$ seconds run for both heuristic algorithms, reporting the average value after $10$ seconds, and the best value reported among the $100$ runs whenever it differs from the average value by more than $1$. The parameters used in SRPL are $ \mu_1= 0.25$, $\mu_2=0.01$.  

In this case, the optimal values are not known, since for the second and fourth graphs we find larger bicliques than the ones reported in \cite{MaximumBicliqueBenchmark}, and also larger than the planted ones. 
From the results, we can see that SRPL outperforms all the other algorithms. In fact, it always manages to find the best value for all graphs, and it is the only algorithm to consistently beat or equate the size of the planted biclique graph.  
Notice that even in the case of low density, the matrix $A$ has so many nonzero entries that Gurobi cannot even move from an initial guess of $0$ for the second graph, and overloads the RAM for the fourth graph.  
\begin{table}[H]
\caption{Numerical comparison for Gur, BFAS, E-AO and SRPL for the problem of finding the maximum edge biclique in four different bipartite graphs.
The table reports the maximum edge biclique found in the time limit (10 seconds) for Gurobi and BFAS. The reported numbers for E-AO and SRPL are the average value found within 10 seconds for 100 runs, and in parentheses the best value found throughout all 100 runs when it differs from the average one. Gurobi cannot be executed on the last graph due to its excessive size.  
\label{table:Gur_BFAS_biclique}}
\begin{center}
\begin{tabular}{c|cccc}
$m\times n$ & $100\times 100 $ & $300\times 300 $& $100\times 100 $& $10000\times 100 $\\ \hline  \hline
Gur& $      \bm{2500} $ & $0$ & $310$ &  NA \\  \hline
BFAS& $       3 $& $2$& $2$& $2$\\ \hline
E-AO& $       66 $ & $\bm{114}$& $87$& $12$\\ \hline
SRPL & $      \bm{ 2500} $& $\bm{114}$& $\bm{6400}$& $46$($\bm{358}$)
\end{tabular} 
 \end{center} \vspace{-.5cm} 
 \end{table}

\subsection{Computing the maximal angle between PSD and symmetric nonnegative cones via Pareto singular values}

 \begin{wrapfigure}{l}{0.4\textwidth}
 \vspace{-1cm}
\begin{center}
\begin{tabular}{c||c|c}
$n$ & block circulant & $\MA(\mca P_n,\mca N_n)$ \\ \hline 
$5$& $0.7575\pi  $ & $0.7575\pi  $ \\  
$6$& \color{blue}$0.7575\pi  $ & $0.7575\pi  $ \\  
$7$& \color{blue}$0.7575\pi  $ & $0.7575\pi  $ \\  
$8$& $0.7608\pi  $ & $0.7608\pi  $ \\ 
$9$& \color{blue}$0.7608\pi  $ & $0.7608\pi  $ \\ 
$10$& \color{blue}$0.7608\pi  $ & $0.7609\pi  $ \\ 
$11$& $0.7627\pi  $ & $0.7627\pi  $ \\ 
$12$& $0.7649\pi  $ & $0.7649\pi  $ \\ 
$13$& \color{blue}$0.7649\pi  $ & $0.7649\pi  $ \\ 
$14$& \color{blue}$0.7649\pi  $ & \color{red}$0.7659\pi  $ \\ 
$15$& \color{blue}$0.7649\pi  $ & \color{red}$0.7678\pi  $ \\ 
$16$& $0.7670\pi  $ & \color{red}$0.7699\pi  $ \\ 
$17$& \color{blue}$0.7670\pi  $ & \color{red}$0.7699\pi  $ \\ 
$18$& \color{blue}$0.7670\pi  $ & \color{red}$0.7699\pi  $ \\ 
$19$& $0.7681\pi  $ & \color{red}$0.7703\pi  $ \\ 
$20$& $0.7719\pi  $ & $0.7719\pi  $ \\ 
$21$& \color{blue}$0.7719\pi  $ & $0.7719\pi  $ \\ 
$22$& \color{blue}$0.7719\pi  $ & $0.7719\pi  $ \\ 
$23$& \color{blue}$0.7719\pi  $ & \color{red}$0.7722\pi  $ 
\end{tabular} 
\captionof{table}{First column: Largest angle between block-circulant matrices in $\mca P_n$ and in $\mca N_n$. Second column: Largest known angle between $\mca P_n$ and $\mca N_n$.
}\label{table:Block_Circulant_PD_SNN_comparison}
\end{center}
 \vspace{-.5cm}
 \end{wrapfigure}

Let $\mathcal S^n$ denote the space of symmetric matrices of order $n$. Let $\mathcal P_n\subset\mathcal S^n$ denote the cone of positive semidefinite matrices (PSD cone), and let $\mathcal N_n\subset \mathcal S^n$ denote the cone of matrices with nonnegative entries. 
For any matrix $M$ denote $M^-:= -\min\{M,0\}$. 
From \cite{Goldberg2013OnTM}, we know that given  a nonzero $N\in \mca N_n$,  
{the matrix $P\in \mca P_n$ that maximizes the angle between $N$ and $\mca P_n$} 
is the negative semidefinite part of $N$, that is, if $N = Q\Lambda Q^\top$ is the eigendecomposition of $N$, then $P =Q\Lambda^-Q^\top = Q(Q^\top NQ)^-Q^\top$. 
Vice versa, given a non zero $P\in \mca P_n$ , 
{the matrix $N\in \mca N_n$ that maximizes the angle between $\mca N_n$ and $P$} is the negative part of $P$, that is, $N = P^-$.  

The problem $\MA(\mca P_n,\mca N_n)$ has been studied by several authors, in particular because it gives a bound on the problem  $\MA(\mca T_n,\mca T_n)$ where $\mca T_n$ is the cone of copositive matrices \cite{Goldberg2013OnTM,Zhang2021matrixangle5x5,Yu2025copositive3x3}. 
The cone $\mca P_n$ is not polyhedral, so in the next section we further restrict the problem to the polyhedral cone of \textit{circulant matrices} inside $\mca P_n$, called $\mca {CP}_n$. We will prove that solving $\MA(\mca {CP}_n,\mca N_n)$ is equivalent to solve $\MA(\mca {CP}_n,\mca {CN}_n)$, where $\mca {CN}_n$ are the circulant matrices in $\mca {N}_n$ (Lemma~\ref{lem:equivalence_circulant_SPN_SNN}). As a consequence, all operations will be done on the algebra of circulant matrices, allowing us to further simplify the problem.

The motivation behind this simplification is that, for many dimensions,  the best value for the problem $\MA(\mca {CP}_m,\mca {CN}_m)$  among all $m\le n$ corresponds to the best known value for the general problem $\MA(\mca P_n,\mca N_n)$, as shown in \cite{de2023computing} and reported in Table~\ref{table:Block_Circulant_PD_SNN_comparison}. In other words, $\MA(\mca P_n,\mca N_n)$ is often solved by a pair of block-circulant matrices, that is,  matrices with diagonal blocks, where all blocks are circulant. 

 Table~\ref{table:Block_Circulant_PD_SNN_comparison} provides the best values for the two problems, where the angles between block-circulant matrices in $\mca P_n$ and those in $\mca N_n$ have been computed exactly by Gur up to dimension $23$ (for higher dimensions, the computational time exceeds $24$ hours).
 The values in blue correspond to solutions that are not circulant, but just block circulant, and they are always equal to $\max_{m\le n} \MA(\mca {CP}_m,\mca {CN}_m)$. In particular, such optimal block-circular matrices will be some optimal couple of circulant matrices for $\MA(\mca {CP}_m,\mca {CN}_m)$ with $m<n$, padded with zero rows and columns to reach dimension $n$. 

The values in red indicate the dimensions where we know a feasible point for $\MA(\mca P_n,\mca N_n)$ with larger angle than the optimal solution of $\MA(\mca {CP}_m,\mca {CN}_m)$ for every $m\le n$, with a tolerance of $10^{-4}\pi$ to account for rounding errors. We see that up to dimension $13$, the problem restricted to block-circulant matrices presents the same optimal angle as the best known one for $\MA(\mca P_n,\mca N_n)$, up to our tolerance level. There are some higher dimensions where the two problems have the same values, but they tend to be less and less frequent as the dimension increases. This is a topic for further research.

\subsubsection{Approaching the maximal angle by symmetric circulant matrices}

\begin{lemma}\label{lem:equivalence_circulant_SPN_SNN}
   The problem $\MA(\mca {CP}_n,\mca {N}_n)$ 
   has the same solutions as $\MA(\mca {CP}_n,\mca {CN}_n)$. If $n=1+2m$ is odd, then $\MA(\mca {CP}_n,\mca {CN}_n)$ has the same solution as  $\PSV(M)$, where
\[
M = \frac 2{\sqrt n}\left[   \cos \left(  \frac {2\pi}nij 
\right)
\right]_{i,j=1:m} 
 \in \mathbb{R}^{m \times m}. 
\] 
\end{lemma}
\begin{proof}
The basic circulant matrix is 
\[
C\in \f R^{n\times n},\quad C_{i,j} = \begin{cases}
    1, & j-i \equiv 1 \pmod n,\\
    0, & \text{otherwise}, 
\end{cases}\quad 
C =
\begin{pmatrix}
     0 & 1 & & \\
     & \ddots & \ddots & \\
      & &  0 & 1\\
     1  & &   & 0
\end{pmatrix}. 
\]
The algebra of real circulant matrices $\f R[C]$ is the set of all matrices $A$ such that the $i$-th row of $A$ is the right cyclically $(i-1)$-shifted of the first row. It is known that $A$ can be decomposed into two forms: through the entries of its first row, and through its eigenvalues. In other words, given $[a_1, a_2, \dots, a_n]$ the first row of $A$, and denoting by $\lambda_1,\ldots,\lambda_n$ the eigenvalues of $A$, then 
\begin{equation*}
A = \sum_{k=1}^n a_k C^{k-1} = \sum_{j=1}^n \lambda_j f_j f_j^H,
\end{equation*}
where $\lambda_j := \sum_k a_k e^{\textnormal i \frac {2\pi}n(j-1)(k-1) }$, and $f_j:=(1/\sqrt{n})\left[e^{\textnormal i \frac {2\pi}n(i-1)(j-1) }\right]_{i=1:n}$ are the columns of the unitary Fourier matrix $F_n$, and where $z^H$ denotes the conjugate transpose of $z$.

The symmetric real circulant matrices are such that $a_{k} = a_{n-k+2}$ for any $k=2,\dots,n$. Moreover necessarily $\lambda_j\in \f R$ for any $j$, and $\lambda_j=\lambda_{n-j+2}$. Since they have the same number of real parameters, the set of symmetric real circulant matrices are identified either by $a_1, a_2, \dots, a_{\lceil(n-1)/2\rceil}$ or by $\lambda_1, \lambda_2, \dots, \lambda_{\lceil(n-1)/2\rceil}$. 

Call now $\mca {CP}_n\subset\mca P_n$ the (convex) cone of positive semidefinite circulant matrices, and $\mca {CN}_n\cu \mca N_n$ the (convex) cone of nonnegative symmetric circulant matrices. 

Notice now that given $N\in \mca{CN}_n$, the matrix $P = F_n(F_n^H N F_n)^{-}F_n^H\in\mca P_n$ maximizes the angle between $N$ and $\mca P_n$, and it is still circulant, so $P\in \mca {CP}_n$. Since the largest eigenvalue of $N$ is the one associated to $f_1$, we find that $\lambda_1(P) = 0$. 
Moreover, given $P\in \mca{PN}_n$, the matrix $N = P^-\in \mca N_n$ maximizes the angle between $P$ and $\mca N_n$, and it is still circulant, so $N\in \mca {CN}_n$. Since the diagonal of $P$ is nonnegative, we find that the diagonal of $N$ is zero, that is, $a_1(N) = 0$. 

Therefore, the problem $\MA(\mca {CP}_n,\mca {CN}_n)$ has the same solution as $\MA(\mca {P}_n,\mca {CN}_n)$ and $\MA(\mca {CP}_n,\mca {N}_n)$. 
It is in particular solved by a couple $(P,N)\in\mca {CP}_n\times\mca {CN}_n$ such that $\lambda_1(P) = 0$ and $a_1(N) = 0$. 

If $n$ is odd and $n=1+2m$, then we can write any real symmetric circulant $A$ as
\[
A = a_1 I + \sum_{k=1}^{m} a_{k+1} (C^{k} + C^{n-k})
= \lambda_1 \frac 1n{\bf 1}{\bf 1}^\top + 2\sum_{j=2}^{m+1} \lambda_{j} \mc R(f_j f_j^H),
\]
where $\mc R(z)$ denotes the real part of the complex matrix $z$.
Any matrix $P\in \mca {CP}_n$ is uniquely identified by $m+1$ of its real and nonnegative eigenvalues $[\lambda_1, \lambda_2, \dots, \lambda_{m+1}]$. Analogously, any $N\in \mca {CN}_n$ is uniquely identified by $m+1$ of the nonnegative elements on its first row $[a_1, a_2, \dots, a_{m+1}]$. 
 Calling $a:=[a_2,\dots,a_{m+1}]^\top$, $\lambda:=[\lambda_2,\dots,\lambda_{m+1}]^\top$, $x = [x_1,\dots,x_{m}]^\top = a/\sqrt{2n}$,  $y = [y_1,\dots,y_{m}]^\top = \lambda/\sqrt{2}$, then
\begin{align*}
    \min_{\substack{(P,N)\in\mca {CP}_n\times\mca {CN}_n\\ \|P\|_F=\|N\|_F = 1}} \tr(PN)
&=
\min_{\substack{a,\lambda \ge 0,\\ \|a\|^2= 1/2n, \|\lambda\|^2 = 1/2}} 
2\tr\left[\left(
\sum_{k=1}^{m} a_{k+1} (C^{k} + C^{n-k})\right)\left(
\sum_{j=2}^{m+1} \lambda_{j} \mc R(f_j f_j^H)
\right)\right]\\
&=
\min_{\substack{a,\lambda \ge 0,\\ \|a\|^2= 1/2n, \|\lambda\|^2 = 1/2}} 
2\sum_{k=1}^{m}\sum_{j=1}^{m}
 a_{k+1}\lambda_{j+1}
\mc R \left[
 f_{j+1}^H(C^{k} + C^{n-k})f_{j+1} 
\right]\\
&=
\min_{\substack{x,y \ge 0,\\ \|x\|^2= 1, \|y\|^2 = 1}} 
\frac{2}{\sqrt{n}}\sum_{k=1}^{m}\sum_{j=1}^{m}
 x_{k}y_{j}
  \cos \left[  \frac {2\pi}njk 
\right],
\end{align*}
thus, $\MA(\mca {CP}_n,\mca {CN}_n)$ has the same solutions as $\PSV(M)$.

\end{proof}

Table \ref{table:Gur_BFAS_circ_PSD_NNS} reports the optimal values found by Gur and BFAS regarding the problem of finding the largest angle between the circulant PSD cone and the cone of nonnegative symmetric circulant matrices in dimension $n$ for $n=13,15,17,19,21,23$.
When the algorithms terminate in less than $60$ seconds, we report the elapsed times, otherwise we report the best values found in $60$ seconds. 
The results are then compared with the exact maximal angle between circulant matrices in $\mca P_n$ and those in $\mca N_n$, computed by Gur without imposing a time limit up to dimension $n=23$ (for higher dimensions, the computational time exceeds $24$ hours).
The numbers in bold are the optimal angle found, 
whenever they coincide with the exact angles up to a tolerance of $10^{-5}\pi$
.

We observe that BFAS outperforms Gur both in speed and accuracy. In fact, even when its computational time exceeds a minute, it finds the exact solution in all dimensions. Moreover, it can solve all problems in less than $20$ seconds up to dimension $21$.
Gur instead can only solve the problem exactly up to dimension $15$, and the approximations at $60$ seconds is only reliable up to dimension $17$.  
Here Gurobi is way slower than in the previous problem probably because the matrix $M$ is dense.

 Table \ref{table:TL10_circ_PSD_NNS} reports a comparison of the four algorithms with a time limit of $10$ seconds for the same problem in dimensions $n=17,19,21,23,25,27$. Again, E-AO and SRPL\footnote{The algorithm that we are using for this case is the SRPL method for cone of matrices developed in \cite[Section\,6.3]{de2023computing}.} usually are restarted until we reach the $10$ seconds mark, and we perform $100$ of these $10$ seconds run for both heuristic algorithms, reporting the average value. Notice that we know the ground truth only up to dimension $23$
 . For higher dimensions, we report the best known values, obtained by letting Gurobi run for at least $24$ hours on each problem. 
 The parameters used by  SRPL in this case are $\mu_1=0.25$, $\mu_2=0.01$.

E-AO and SRPL are the only methods to converge to the optimum value within the $10$ seconds for all dimensions. Gur is not reliable already from dimension $19$, and BFAS starts to give incorrect results from dimension $25$. 
All algorithms 
reach the optimal value $0.766370 \pi$ in the case $n=23$ in less than $0.02$s except for BFAS that takes more than $10$ seconds. 

\begin{table}[H]
\caption{Numerical comparison of Gurobi and BFAS for different dimensions for the problem of finding the maximal angle between the PSD cone and the nonnegative symmetric cone, both restricted to the subalgebra of circulant matrices. The table reports the optimal objective function values found in the time limit (60 seconds) and the actual elapsed time. We also report the exact value for each problem.
\label{table:Gur_BFAS_circ_PSD_NNS}}
\begin{center}
\begin{tabular}{c|cccccc}
$n$ & $13 $& $15 $& $17 $& $19 $& $21 $& $23 $\\ \hline  \hline 
exact&   $0.762950 \pi$& $0.757765 \pi$& $0.764971 \pi$& $0.768062 \pi$& $0.768769 \pi$& $0.766370 \pi$\\ \hline  \hline 
Gur&   $ \bm{0.762950} \pi$& $ \bm{0.757765} \pi$& $ \bm{0.764971} \pi$& $ 0.767876 \pi$& $ 0.765409 \pi$& $ \bm{0.766370} \pi$\\ 
&  $   0.854 $ s& $  25.061 $ s& $60 ^*$ s& $60 ^*$ s& $60 ^*$ s& $60 ^*$ s\\ \hline
BFAS&   $ \bm{0.762950} \pi$& $ \bm{0.757765}\pi$& $ \bm{0.764971} \pi$& $\bm{ 0.768062} \pi$& $ \bm{0.768768} \pi$& $ \bm{0.766370} \pi$\\ 
&   $   \bm{0.333} $ s& $   \bm{0.356} $ s& $   \bm{1.114} $ s& $   \bm{4.418} $ s& $  \bm{19.953} $ s& $60 ^*$ s
\end{tabular} 
 \end{center}
 \end{table}

\begin{table}[H]
\caption{Numerical comparison of Gur, BFAS, E-AO and SRPL for different dimensions for the problem of finding the maximal angle between the PSD cone and the nonnegative symmetric cone, both restricted to the subalgebra of circulant matrices. 
The table reports the optimal objective function values found in the time limit (10 seconds). The reported numbers for E-AO and SRPL are the best values found after 10 seconds for 100 runs. We also report, when available, the exact value for each problem, and the best known lower bound when the exact value is not available, indicated with an asterisk.
\label{table:TL10_circ_PSD_NNS}}
\begin{center}
\begin{tabular}{c|cccccc}
$n$ &   $17 $& $19 $& $21 $& $23 $& $25 $& $27 $\\ \hline  \hline 
exact&  $0.764971 \pi$& $0.768062 \pi$& $0.768769 \pi$& $0.766370 \pi$& $0.767385 \pi^*$ &  $0.768258 \pi^*$  \\ \hline  \hline 
Gur&  $ \bm{0.764971} \pi$& $ 0.759309 \pi$& $ 0.765409 \pi$& $ \bm{0.766370}\pi$& $\bm{0.767385} \pi$& $ 0.760879 \pi$\\ \hline
BFAS&  $ \bm{0.764971}\pi$& $ \bm{0.768062} \pi$& $ \bm{0.768768} \pi$& $ \bm{0.766370 }\pi$& $ 0.762620 \pi$& $ 0.756841 \pi$\\ \hline
E-AO&  $ \bm{0.764971} \pi$& $ \bm{0.768062} \pi$& $ \bm{0.768768} \pi$& $ \bm{0.766370} \pi$& $ \bm{0.767385}\pi$& $ \bm{0.768258} \pi$\\ \hline
SRPL&  $\bm{ 0.764970} \pi$& $\bm{0.768062} \pi$& $ \bm{0.768768} \pi$& $ \bm{0.766369 }\pi$& $\bm{0.767384} \pi$& $ \bm{0.768257} \pi$
\end{tabular} 
 \end{center}
 \end{table}

\paragraph{Computing $\MA(\mca P_n,\mca N_n)$ with E-AO and SRPL} 

If we consider the harder problem $\MA(\mca P_n,\mca N_n)$, forgetting about the restriction to circulant matrices, we cannot use either Gur or BFAS. 
However, E-AO and SRPL can easily be adapted into solving this problem, since it is always possible to compute exactly the projection of a matrix $A$ on the cones $\mca P_n$ and $\mca N_n$~\cite{higham1988computing}. 

Table~\ref{table:EAO_SRPL_PSD_NNS} reports the optimal values found by the modified E-AO and SRPL regarding the problem of finding the largest angle between the PSD cone and the cone of nonnegative symmetric matrices in dimension $n$ for $n=20,30,40,50,60$. We test the algorithms on $1000$ random starting points and we report the best values found (E-AO$_b$, SRPL$_b$) and the average ones (E-AO$_a$, SRPL$_a$). We also report the average elapsed time over the $1000$ runs, its standard deviation, and the best known values for each problem, taken from \cite{de2023computing}. 
The parameters used by  SRPL in this case are $\mu_1=0.1$, $\mu_2=5$.
\begin{table}[h!]
\caption{Numerical comparison for E-AO and SRPL for different dimensions for the problem of finding the maximal angle between the PSD cone and the nonnegative symmetric cone. The table reports the best and average value found over 10000 random initializations, together with the average elapsed time. We also report the best known value for each dimension.
\label{table:EAO_SRPL_PSD_NNS}}
\begin{center}
\begin{tabular}{c|ccccc}
$n$ \hspace{-0.2cm} & \hspace{-0.2cm}  $20 $\hspace{-0.2cm} & \hspace{-0.2cm}  $30 $\hspace{-0.2cm} & \hspace{-0.2cm}  $40 $\hspace{-0.2cm} & \hspace{-0.2cm}  $50 $\hspace{-0.2cm} & \hspace{-0.2cm}  $60 $\\ \hline  \hline 
best known & \hspace{-0.2cm}  $0.7719 \pi$\hspace{-0.2cm} & \hspace{-0.2cm}  $0.7757 \pi$\hspace{-0.2cm} & \hspace{-0.2cm}  $0.7789 \pi$\hspace{-0.2cm} & \hspace{-0.2cm}  $0.7812 \pi$\hspace{-0.2cm} & \hspace{-0.2cm}  $0.7837 \pi$\\ \hline  \hline 
E-AO$_b$\hspace{-0.2cm} & \hspace{-0.2cm}  $\bm{0.7719} \pi$\hspace{-0.2cm} & \hspace{-0.2cm}  $ \bm{0.7757} \pi$\hspace{-0.2cm} & \hspace{-0.2cm}  $\bm{0.7789} \pi$\hspace{-0.2cm} & \hspace{-0.2cm}  $\bm{0.7813} \pi$\hspace{-0.2cm} & \hspace{-0.2cm}  $\bm{0.7837} \pi$\\ 
E-AO$_a$\hspace{-0.2cm} & \hspace{-0.2cm}  $ 0.7697 \pi$\hspace{-0.2cm} & \hspace{-0.2cm}  $ 0.7741 \pi$\hspace{-0.2cm} & \hspace{-0.2cm}  $ 0.7768 \pi$\hspace{-0.2cm} & \hspace{-0.2cm}  $ 0.7790 \pi$\hspace{-0.2cm} & \hspace{-0.2cm}  $ 0.7805 \pi$\\ 
\hspace{-0.2cm} & \hspace{-0.2cm}  $\bm{   0.022} \pm    0.013$s\hspace{-0.2cm} & \hspace{-0.2cm}  ${   0.111} \pm    0.054$s\hspace{-0.2cm} & \hspace{-0.2cm}  ${   0.701} \pm    0.235$s\hspace{-0.2cm} & \hspace{-0.2cm}  ${   1.263} \pm    0.273$s\hspace{-0.2cm} & \hspace{-0.2cm}  ${   2.852} \pm    0.312$s\\ \hline
SRPL$_b$\hspace{-0.2cm} & \hspace{-0.2cm}  $\bm{0.7719} \pi$\hspace{-0.2cm} & \hspace{-0.2cm}  $\bm{0.7757} \pi$\hspace{-0.2cm} & \hspace{-0.2cm}  $\bm{0.7789} \pi$\hspace{-0.2cm} & \hspace{-0.2cm}  $\bm{0.7812} \pi$\hspace{-0.2cm} & \hspace{-0.2cm}  $\bm{0.7837} \pi$\\ 
SRPL$_a$\hspace{-0.2cm} & \hspace{-0.2cm}  $ 0.7695 \pi$\hspace{-0.2cm} & \hspace{-0.2cm}  $ 0.7739 \pi$\hspace{-0.2cm} & \hspace{-0.2cm}  $ 0.7766 \pi$\hspace{-0.2cm} & \hspace{-0.2cm}  $ 0.7787 \pi$\hspace{-0.2cm} & \hspace{-0.2cm}  $ 0.7802 \pi$\\ 
\hspace{-0.2cm} & \hspace{-0.2cm}  $   0.026 \pm    0.012$s\hspace{-0.2cm} & \hspace{-0.2cm}  $   \bm{0.062} \pm    0.025$s\hspace{-0.2cm} & \hspace{-0.2cm}  $   \bm{0.155} \pm    0.060$s\hspace{-0.2cm} & \hspace{-0.2cm}  $   \bm {0.319} \pm    0.130$s\hspace{-0.2cm} & \hspace{-0.2cm}  $   \bm {0.565} \pm    0.229$s
\end{tabular} 
 \end{center}
 \end{table}

We observe that both algorithms manage to correctly find the best known bound in all cases. E-AO average result is always slightly larger than the SRPL average result, but with a consistently larger computational time per run, and the difference between the average runtimes gets larger with larger dimensions. Already for dimension $60$, SRPL is $5$ times faster than E-AO, and manages to find the same best objective. 
Sometimes, the algorithms may find larger angles than the best known angles, but it has to be attributed to rounding errors.

\section{Conclusion} \label{sec:concl}

In this paper, we studied the concept of singular values of a rectangular matrix, $A$, relative to a pair of closed convex cones, $P$ and $Q$. We also considered two restricted variants: (1)~$A$ is the identity which corresponds to the problem of computing the maximal angle between the cones $P$ and $Q$, and 
(2)~$P$ and $Q$ are the nonnegative orthant which corresponds to the so-called Pareto singular values of $A$. 
We first show that all these problems are NP-hard, while also identifying cases when such problems can be solved in polynomial time. 
Then we proposed 4 algorithms to compute the minimum singular values: two are exact, namely BFAS relying on enumeration and Gur using Gurobi, and two are heuristics, namely E-AO using alternating optimization and SRPL using fractional programming. We then applied these algorithms for various applications. 
Interestingly, there is no clear winner between the four proposed algorithms: each algorithm outperforms the others in at least one of the applications.

\section*{Acknowledgments}  We are grateful to the anonymous reviewers who
carefully read the manuscript, their feedback helped us improve our paper.

\bigskip
\bigskip

\noindent{\bf Acknowledgments} G. Barbarino and N. Gillis acknowledge the support by the European Union (ERC consolidator, eLinoR, no 101085607). D. Sossa acknowledges the support of MATH-AMSUD 23-MATH-09 (AMSUD230018) MORA-DataS project, and the support of the Universidad de O'Higgins through grants MOVI2403 and PNTE2502.

\bigskip
\noindent{\bf Data availability} All data generated or analyzed during this study are included in this article.

\section*{Declarations}
\noindent{\bf Ethical Approval and Consent to participate} All the authors gave the ethical approval and consent to participate in this article.

\bigskip
\noindent{\bf Consent for publication} All the authors gave consent for the publication of identifiable details to be published in the journal and article.

\bigskip
\noindent{\bf Code availability} The codes used in this work are available in\\ \url{https://github.com/giovannibarbarino/coneSV/}.

\bigskip
\noindent{\bf Competing interests} The authors declare no competing interests.

\bibliographystyle{spmpsci}
\bibliography{Article2023}

\end{document}